\documentclass{article}
\usepackage{amsmath,amssymb,amsthm, amsfonts, mathrsfs}
\usepackage{enumerate}
\usepackage{calligra,indentfirst,epsfig}
\usepackage[flushleft]{threeparttable} 
\usepackage{booktabs,caption}
\calligra
\newtheorem{thm}{Theorem}[section]
\newtheorem{lem}{Lemma}[section]
\newtheorem{cor}{Corollary}[section]
\newtheorem{prop}{Proposition}[section]

\theoremstyle{remark}

\newtheorem{ex}{Example}[section]
\theoremstyle{definition}

\allowdisplaybreaks
\begin{document}
\title{\bf {Repeated-root constacyclic codes over finite \\commutative chain rings and their distances}}
\author{ Anuradha Sharma{\footnote{Corresponding Author, Email address: anuradha@iiitd.ac.in} and Tania Sidana}\\
{\it Department of Mathematics, IIIT-Delhi}\\{\it New Delhi 110020, India}}
\date{}
\maketitle
\begin{abstract} Let $\mathcal{R}_e$ be a finite commutative chain ring with nilpotency index $e \geq 2.$ In this paper, all repeated-root constacyclic codes of arbitrary lengths  over $\mathcal{R}_{2},$ their sizes and their dual codes are determined.  As an application, some isodual constacyclic codes  over $\mathcal{R}_{2}$ are also listed. Moreover, Hamming distances, Rosenbloom-Tsfasman distances and Rosenbloom-Tsfasman weight distributions of all repeated-root constacyclic codes over $\mathcal{R}_{2}$ and some repeated-root constacyclic codes over $\mathcal{R}_{e}$ are determined. \end{abstract}
 \vspace{-6mm}\section{Introduction}\label{intro}\vspace{-2mm}
Constacyclic codes over finite fields are introduced and studied by Berlekamp \cite{berl} in 1960's. These codes are generalizations of cyclic and negacyclic codes, and can be effectively encoded and decoded using shift registers.  In 1990's, Calderbank et al. \cite{calder}, Hammons et al.  \cite{hammons} and Nechaev \cite{nech} related many non-linear codes over finite fields  to linear codes over the ring $\mathbb{Z}_{4}$ of integers modulo 4 with the help of a Gray map.  This motivated many researchers to study codes over $\mathbb{Z}_{4}$ in particular and to study codes over  finite commutative chain rings in general.   

Towards this, Abualrub and Oehmke \cite{abu} studied all cyclic codes of length $2^s$ over $\mathbb{Z}_{4},$ where $s$ is a positive integer. Later, Dinh and L$\acute{o}$pez-Permouth \cite{dinh} established algebraic structures of cyclic and negacyclic codes of length $n$ over a finite commutative chain ring $R$ and their dual codes,   provided the length $n$ and the characteristic of the residue field of $R$ are coprime. Besides this, they determined all negacyclic codes of length $2^s$ over the ring $\mathbb{Z}_{2^m}$ of integers modulo $2^m$ and their dual codes, where $s\geq 1$ and  $m \geq 2$ are  integers. In a related work,  Kiah et al. \cite{kiah} determined all repeated-root cyclic codes of length $p^s$ over Galois rings with nilpotency index 2, where $p$ is a prime. Later, Sobhani and Esmaeili \cite{sobh} studied all repeated-root cyclic and negacyclic codes of arbitrary lengths over Galois rings with nilpotency index 2. In a  related work, Batoul et al. \cite{bat}  proved that  repeated-root $\lambda$-constacyclic  codes over a finite chain ring $R$ are equivalent to cyclic codes when $\lambda$ is an $n$th power of a unit in  $R.$  
 
 Next let $p$ be a prime, $s,m$ be positive integers, and let $\mathbb{F}_{p^m}$ be the finite field of order $p^m.$  Dinh \cite{dinh4} obtained all constacyclic codes of length $p^s$ over the finite commutative chain ring $\mathbb{F}_{p^m}[u]/\langle u^2 \rangle$ and their Hamming distances.  Chen et al. \cite{chen}, Dinh et al. \cite{dinh5} and Liu et al. \cite{liu} determined all constacyclic codes of length $2p^s$ over the ring $\mathbb{F}_{p^m}[u]/\langle u^2 \rangle,$ where $p$ is an odd prime.   Later, Sharma and Rani \cite{sharma1, sharma2} determined  all constacyclic codes of length $4p^s$ over $\mathbb{F}_{p^m}[u]/\langle u^2 \rangle,$ their sizes and their dual codes,  where $p$ is an odd prime and $s,m$ are positive integers. They also listed some isodual constacyclic codes of length $4p^s$ over $\mathbb{F}_{p^m}[u]/\langle u^2 \rangle.$ Using a different approach, Cao et al. \cite{cao} determined  all repeated-root $\lambda$-constacyclic codes of arbitrary lengths  over $\mathbb{F}_{p^m}[u]/\langle u^2 \rangle$ and their dual codes by writing a canonical form decomposition for each code, where $\lambda $ is a non-zero element of $\mathbb{F}_{p^m}.$ Later, Zhao et al. \cite{zhao} determined   all repeated-root $(\alpha+\beta u)$-constacyclic codes of arbitrary lengths over $\mathbb{F}_{p^m}[u]/\langle u^2 \rangle$ and their dual codes, where $\alpha,\beta$ are non-zero elements of $\mathbb{F}_{p^m}.$ Thus the algebraic structure of all repeated-root constacyclic codes of arbitrary lengths over $\mathbb{F}_{p^m}[u]/\langle u^2 \rangle$ is established. Recently, Sharma and Sidana \cite{sharma3} determined all repeated-root constacyclic codes of arbitrary lengths over the finite commutative chain ring $\mathbb{F}_{p^m}[v]/\langle v^3 \rangle,$ their sizes and their dual codes. 
In a subsequent work, Dinh et al. \cite{dinh6} studied repeated-root $(\lambda_0^{p^s}+\gamma w)$-constacyclic codes of length $p^s$ over  a finite commutative chain ring $R$ with the maximal ideal as $\left<\gamma \right>,$ where $p$ is a prime number, $s \geq 1$ is an integer and $\lambda_0, w$ are units in $R.$ The constraint that $w$ is a unit in $R$ restricts their study to only a few special classes of constacyclic codes of length $p^s$ over $R.$ The codes belonging to these special classes can be viewed as ideals of the finite commutative chain ring  $R[x]/\langle x^{p^s}-\lambda_0^{p^s}-\gamma w\rangle,$ and hence are principal ideals. On the other hand, when $w$ is a non-unit in $R,$ the ring  $R[x]/\langle x^{p^s}-\lambda_0^{p^s}-\gamma w\rangle$ is a non-chain ring and has non-principal ideals.  In a related work,  Dinh et al. \cite[Prop. 6.3-6.5]{dinh8} determined algebraic structures of  all  $(4z-1)$-constacyclic codes of length $2^s$ over $GR(2^e, m)$ and their dual codes, where  $z$ is a unit in $GR(2^e,m),$ (throughout this paper, $GR(2^e,m)$ denotes the Galois ring of characteristic $2^e$ and cardinality $2^{em}$).  He also determined their Hamming, Homogenous and Rosenbloom-Tsfasman distances, and their Rosenbloom-Tsfasman weight distributions.  However, we noticed an error in Proposition 6.5 of Dinh et al.  \cite{dinh8}  on Rosenbloom-Tsfasman weight distributions, which is illustrated in Example \ref{ex} and is rectified in Theorem  \ref{6thmRT}.

Throughout this paper, let $\mathcal{R}_{e}$ be the finite commutative chain ring with nilpotency index $e \geq 2.$ The main goal of this paper is to determine all repeated-root constacyclic codes of arbitrary lengths  over $\mathcal{R}_{2},$ their sizes,  their dual codes, their Hamming and Rosenbloom-Tsfasman distances, and Rosenbloom-Tsfasman weight distributions.  We also determine algebraic structures, Hamming distances, Rosenbloom-Tsfasman distances and Rosenbloom-Tsfasman weight distributions of some repeated-root constacyclic codes of length $np^s$ over $\mathcal{R}_{e}.$  

This paper is organized as follows: In Section \ref{prelim}, we state some basic definitions and  results  that are needed to derive our main results. In Section \ref{sec3}, we determine all repeated-root constacyclic codes of arbitrary lengths  over $\mathcal{R}_2,$ their sizes and their dual codes (Propositions \ref{p1}, \ref{p2} and Theorems \ref{thm1}, \ref{thm2} -\ref{dual}).  We also determine  their Hamming and Rosenbloom-Tsfasman distances, and Rosenbloom-Tsfasman weight distributions (Theorems \ref{dthm1}-\ref{dthmRw}, \ref{dthm2}-\ref{RTW2}). As an application of these results, we  obtain some isodual constacyclic codes of arbitrary lengths over $\mathcal{R}_2$ (Corollaries \ref{c1} and \ref{c2}).  In Section \ref{sec4},  we  determine  Hamming distances, Rosenbloom-Tsfasman distances and Rosenbloom-Tsfasman weight distributions of some repeated-root constacyclic codes of length $np^s$ over $\mathcal{R}_{e}$ for any integer $e \geq 2$ (Theorems \ref{6thm1}-\ref{6thmRT}). 
\vspace{-6mm}\section{Some preliminaries}\label{prelim}\vspace{-3mm}
Let $R$ be a finite commutative ring with unity,  $N$ be a positive integer,  $R^N$ be the $R$-module consisting of all $N$-tuples over $R,$ and let $\lambda $ be a unit in $R.$ Then a $\lambda$-constacyclic code $\mathcal{C}$ of length $N$ over $R$ is defined as an $R$-submodule of $R^N$ satisfying the following: $(a_0,a_1,\cdots,a_{N-1})\in \mathcal{C}$ implies that $(\lambda a_{N-1},a_0,a_1,\cdots,a_{N-2}) \in \mathcal{C}.$  An important parameter of the code $\mathcal{C}$ is its  Hamming distance,  which  is a measure of its error-detecting and error-correcting capabilities. The Hamming distance of $\mathcal{C},$ denoted by $d_H(\mathcal{C}),$ is defined as $d_H(\mathcal{C})=\min\{w_H(c): c \in \mathcal{C}\setminus \{0\}\},$ where  $w_H(c)$ equals the number of non-zero components of $c$ and is called the Hamming weight of $c.$   Another important metric in coding theory is the Rosenbloom-Tsfasman (RT) metric, which has applications in uniform distributions. The Rosenbloom-Tsfasman (RT) distance  of the code $\mathcal{C},$ denoted by $d_{RT}(\mathcal{C}),$ is defined as $d_{RT}(\mathcal{C})=\min\{ w_{RT}(c) | c\in \mathcal{C} \setminus \{ 0\}\},$ where  $w_{RT}(c)$ is the RT weight of $c$ and is defined as  \vspace{-2mm}\begin{equation*}w_{RT}(c)=\left\{\begin{array}{ll} 1+ \max \{ j | c_j \neq 0\} & \text{if } c=(c_0,c_1,\cdots,c_{N-1})  \neq 0; \\ 0 & \text{if }c=0.\end{array}\right.\vspace{-2mm}\end{equation*} The Rosenbloom-Tsfasman (RT)  weight distribution of $\mathcal{C}$ is defined as the list $\mathcal{A}_{0}, \mathcal{A}_{1},\cdots,\mathcal{A}_{N},$ where for $0 \leq \rho \leq N,$ $\mathcal{A}_{\rho}$ denotes the number of codewords in $\mathcal{C}$ having the RT weight as $\rho.$

The dual code of $\mathcal{C},$ denoted by $\mathcal{C}^{\perp},$ is defined as $\mathcal{C}^{\perp}=\{u \in R^N: u.c=0 \text{ for all }c \in R^N\},$ where $u.c=u_0c_0+u_1c_1+\cdots + u_{N-1}c_{N-1}$ for $u=(u_0,u_1,\cdots,u_{N-1})$ and $c=(c_0,c_1,\cdots,c_{N-1})$ in $R^N.$ One can observe that $\mathcal{C}^{\perp}$ is a $\lambda^{-1}$-constacyclic code of length $N$ over $R.$  We say that the code $\mathcal{C}$ is isodual if it is $R$-linearly equivalent to its dual code. 
Further, under the $R$-module isomorphism  from $ R^N $ onto $R[x]/\langle x^N-\lambda\rangle,$ defined as $(a_0,a_1,\cdots,a_{N-1}) \mapsto a_0+a_1 x+\cdots+a_{N-1}x^{N-1}+\langle x^N-\lambda \rangle$ for each $(a_0,a_1,\cdots,a_{N-1})\in R^N,$ the code $\mathcal{C}$ can also  be viewed as an ideal of the quotient ring $R[x]/\langle x^N-\lambda\rangle.$ In the light of this, the study of $\lambda$-constacyclic codes of length $N$ over $R$ is equivalent to the study of ideals of the quotient ring $R[x]/\langle x^N-\lambda\rangle.$ From now on, all the elements of $R[x]/\langle x^N-\lambda \rangle$ shall be represented by their representatives in $R[x]$ of degree less than $N,$ and their addition and multiplication shall be performed  modulo $x^N-\lambda.$ In view of this,  the Hamming weight $w_H(c(x))$ of $c(x) \in R[x]/\langle x^N-\lambda \rangle$ is defined  as the number of non-zero coefficients of $c(x),$ while the RT weight $w_{RT}(c(x))$ of $c(x) \in R[x]/\langle x^N-\lambda \rangle$ is defined as $w_{RT}(c(x))=\left\{\begin{array}{ll} 1+\text{deg }c(x) & \text{if }c(x) \neq 0;\\ 0 & \text{if }c(x)=0,\end{array}\right.$  (throughout this paper, $\text{deg }f(x)$ denotes the degree of a non-zero polynomial $f(x) \in R[x]$). Further, it is easy to observe that the dual code $\mathcal{C}^{\perp}$ of $\mathcal{C}$  is given by $\mathcal{C}^{\perp}=\{u(x) \in  R[x]/\langle x^N-\lambda^{-1}\rangle: u(x)c^*(x)=0 \text{ in } R[x]/\langle x^N-\lambda^{-1}\rangle \text{ for all } c(x) \in \mathcal{C}\},$ where $c^*(x)=x^{\text{deg }c(x)}c(x^{-1})$ for all $c(x) \in \mathcal{C}\setminus\{0\}$ and $c^*(x)=0$ if $c(x)=0.$ The annihilator of $\mathcal{C}$ is defined as $\text{ann}(\mathcal{C})=\{f(x) \in R[x]/\langle x^N-\lambda\rangle: f(x) c(x)=0 \text{ in }R[x]/\langle x^N-\lambda\rangle \text{ for all } c(x) \in \mathcal{C}\},$ which is clearly an  ideal of $R[x]/\langle x^N-\lambda\rangle.$ Further, if $I$ is an ideal  of $R[x]/\langle x^N-\lambda \rangle,$ then we define $I^*=\{f^*(x): f(x) \in I\},$ where $f^*(x)= x^{\text{deg }f(x)}f(x^{-1})$ if $f(x) \neq 0$ and $f^*(x)=0$ if $f(x)=0.$ It is easy to see that $I^*$ is an ideal of the ring $R[x]/\langle x^N-\lambda^{-1}\rangle.$ Now the following result is well-known.                       
\vspace{-2mm}\begin{lem}\cite{chen} \label{pr2}Let $\mathcal{C} \subseteq R[x]/\langle x^N-\lambda\rangle$ be a $\lambda$-constacyclic code of length $N$ over $R.$ Then we have $\mathcal{C}^{\perp}=\text{ann}(\mathcal{C})^*.$\vspace{-2mm}\end{lem}
A commutative ring with unity is called  (i) a local ring if  it has a unique maximal ideal and (ii) a chain ring if  all its ideals form a chain with respect to the  inclusion relation.  Then the following hold.
\vspace{-2mm}\begin{prop}\label{pr1}\cite{dinh} For a finite commutative ring $R$ with unity, the following statements are equivalent:
\vspace{-2mm}\begin{enumerate}\item[(a)] $R$ is a local ring and the (unique) maximal ideal $\mathcal{M}$ of $R$ is principal,  i.e., $\mathcal{M}=\langle\gamma \rangle$ for some $\gamma \in R.$
\vspace{-2mm}\item[(b)] $R$ is a local principal ideal ring.
\vspace{-2mm}\item[(c)] $R$ is a chain ring  and all its ideals are given by $\{0\}, R,$ $ \langle \gamma \rangle,  \langle\gamma^2\rangle, \cdots, \langle\gamma^{e-1}\rangle,$ where $e$ is the nilpotency index of $\gamma.$ Moreover, if $\overline{R}= R/\langle\gamma \rangle,$ then $\overline{R}$ is a finite field (called the residue field of $R$) and $|\langle\gamma^\ell \rangle|=|\overline{R}|^{e-\ell}$ for $0 \leq \ell \leq e.$ (Throughout this paper, $|A|$ denotes the cardinality of the set $A.$)
\end{enumerate}  \vspace{-5mm}\end{prop}
From now on, throughout this paper, let $\mathcal{R}_e$ be a finite commutative chain ring  with unity $1$ and with the maximal ideal (and hence the  nil radical) as $\mathcal{M}=\langle\gamma \rangle,$ where $e \geq 2$ is the nilpotency index of  the generator $\gamma$ of $\mathcal{M}.$ The ring $\mathcal{R}_{e}$ is called a finite commutative chain ring with nilpotency index $e.$ Next let $\overline{\mathcal{R}}_{e}=\mathcal{R}_{e}/\langle\gamma \rangle$ be the residue field of $\mathcal{R}_{e}.$ As $\overline{\mathcal{R}}_{e}$ is a finite field, $\text{char }\overline{\mathcal{R}}_{e}$ is a prime number, say $p.$ Let us suppose that $|\overline{\mathcal{R}}_{e}|=p^m$ for some positive integer $m.$ 
 \vspace{-2mm}\begin{prop}  \label{teich}\cite{dinh6, mcdon}  We have the following:
\begin{enumerate}\vspace{-2mm}\item[(a)] The characteristic of  $\mathcal{R}_{e}$ is $p^a,$ where $1 \leq a \leq e.$ Moreover, $|\mathcal{R}_{e}|=|\overline{\mathcal{R}_{e}}|^{e}=p^{me}.$ \vspace{-2mm}\item[(b)] There exists an element $\zeta \in \mathcal{R}_{e}$ having the multiplicative order as  $p^{m}-1.$  Moreover,  each element $r \in \mathcal{R}_{e}$ can be uniquely expressed as $r=r_0+ r_1 \gamma + r_2 \gamma^2+\cdots + r_{e-1} \gamma ^{e-1},$ where $r_i \in \mathcal{T}_{e}= \{ 0,1, \zeta,\cdots,\zeta^{p^{m}-2} \}$ for $0 \leq i \leq e-1.$ (The set $\mathcal{T}_{e}= \{ 0,1, \zeta,\cdots,\zeta^{p^{m}-2} \}$ is called the Teichm\"{u}ller set of $\mathcal{R}_{e}.$)   \vspace{-2mm}\item[(c)]  Let $r =r_0+ r_1 \gamma + r_2 \gamma^2+\cdots + r_{e-1} \gamma ^{e-1},$ where $r_i \in \mathcal{T}_{e}= \{ 0,1, \zeta,\cdots,\zeta^{p^{m}-2} \}$ for $0 \leq i \leq e-1.$ Then $r$ is a unit in $\mathcal{R}_{e}$ if and only if $r_0 \neq 0.$  Moreover, there exists $\alpha_0 \in \mathcal{T}_{e}$ satisfying $\alpha_0^{p^s}=r_0.$ \vspace{-1mm}\end{enumerate}
\vspace{-3mm}\end{prop}
\vspace{-1mm}Let $^{-}: \mathcal{R}_{e} \rightarrow \overline{\mathcal{R}}_{e}$ be the natural epimorphism from $\mathcal{R}_{e}$ onto $\mathcal{R}_{e}/\langle\gamma \rangle,$ which is given by $r \mapsto \overline{r}=r +\langle\gamma \rangle$ for each $r \in \mathcal{R}_{e}.$ Note that  $\overline{\mathcal{R}}_{e}=\{ 0,\overline{1}, \overline{\zeta},\cdots,\overline{\zeta}^{p^{m}-2} \}.$ The map $^{-}$ can be further extended to a ring epimorphism from $\mathcal{R}_{e}[x]$ onto $\overline{\mathcal{R}}_{e}[x]$ as follows:   $\sum\limits_{i=0}^{k} b_ix^i \mapsto \sum\limits_{i=0}^{k}\overline{b_i}x^i$ for every $\sum\limits_{i=0}^{k} b_ix^i \in \mathcal{R}_{e}[x].$ A polynomial $f(x) \in \mathcal{R}_{e}[x]$ is said to be basic irreducible over $\mathcal{R}_{e}$ if $\overline{f(x)}$ is irreducible over $\overline{\mathcal{R}}_{e}.$ 
 Two polynomials $k_1(x), k_2(x)\in \mathcal{R}_{e}[x]$ are said to be coprime  if $\langle k_1(x) \rangle+\langle k_2(x)\rangle=\mathcal{R}_{e}[x],$ i.e., if there exist polynomials $a_1(x), a_2(x)\in \mathcal{R}_{e}[x]$  such that $k_1(x)a_1(x)+k_2(x)a_2(x)=1.$ In general, the polynomials $k_1(x), k_2(x),\cdots, k_r(x) \in \mathcal{R}_{e}[x]$ are said to be pairwise coprime  in $\mathcal{R}_{e}[x]$ if for each $i \neq \ell,$  $k_i(x)$ and $k_\ell(x)$  are coprime in $\mathcal{R}_{e}[x].$  
\vspace{-2mm}\begin{lem}\label{coprime}\cite{norton} The following hold.\vspace{-1mm} \begin{enumerate}\vspace{-2mm}\item[(a)]  Let $k_1(x), k_2(x)\in \mathcal{R}_{e}[x].$ Then $k_1(x)$ and $ k_2(x)$ are coprime in $\mathcal{R}_{e}[x]$ if and only if $\overline{ k_1(x)}$ and $\overline{k_2(x)}$ are coprime in $\overline{\mathcal{R}}_{e}[x].$
\vspace{-2mm}\item[(b)] Let $f(x) \in \mathcal{R}_{e}[x]$ be a monic polynomial such that $\overline{f(x)}$ is square-free, i.e., $\overline{f(x)}$ is not divisible by the square of any irreducible polynomial over $\overline{\mathcal{R}}_{e}.$ Then $f(x)$ factors uniquely as a product of monic basic irreducible pairwise coprime polynomials in $\mathcal{R}_{e}[x].$
\end{enumerate}\vspace{-5mm}\end{lem}

Next to study the algebraic structures of all repeated-root constacyclic codes over $\mathcal{R}_{e},$ we need the following divisibilty result from number theory.
\vspace{-2mm}\begin{prop}\cite{dinh6}\label{bino} Let $p$ be a prime number, and let $b \geq 1,$ $\ell \geq k \geq 0$ be integers. Then the following hold.\vspace{-1mm}
\begin{enumerate}
\vspace{-2mm}\item[(a)] If $p^\ell>k$ and $p^k || b,$ then $p^{\ell-k}||\binom{p^\ell}{b}.$
\vspace{-2mm}\item[(b)] For each integer $i$ satisfying $1\leq i \leq p-1,$ we have $p||\binom{p^\ell}{ip^{\ell-1}}.$ (Throughout this paper, by $p^k|| b,$ we mean $p^k|b $ but $p^{k+1} \nmid b.$)
\end{enumerate} 
\vspace{-4mm}\end{prop}
The following theorem is  an extension of Theorem 3.4 of Dinh \cite{dinh4} and is useful in the determination of  Hamming distances of repeated-root constacyclic codes over $\mathcal{R}_{e}.$

\vspace{-2mm}\begin{thm}\label{dthm} For $\eta  \in \mathbb{F}_{p^m} \setminus \{0\},$ there exists $\eta_0 \in \mathbb{F}_{p^m}$ satisfying $\eta=\eta_0^{p^s}.$ Further, suppose that the polynomial $x^n-\eta_0$ is irreducible over $\mathbb{F}_{p^m}.$ Let $\mathcal{C}$ be an $\eta$-constacyclic code of length $np^s$ over $\mathbb{F}_{p^m}.$ Then we have $\mathcal{C}= \langle (x^n-\eta_0)^{\upsilon} \rangle ,$ where $0 \leq \upsilon \leq p^s.$ Furthermore, the  Hamming distance $d_H(\mathcal{C})$ of the code $\mathcal{C}$ is given by 
\vspace{-3mm}\begin{equation*}d_H(\mathcal{C})=\left\{\begin{array}{ll}
1 & \text{ if } \upsilon=0;\\
\ell+2 & \text{ if } \ell p^{s-1}+1 \leq \upsilon \leq (\ell +1)p^{s-1} \text{ with } 0 \leq \ell \leq p-2;\\
(i+1)p^k & \text{ if } p^s-p^{s-k}+(i-1)p^{s-k-1}+1\leq \upsilon \leq p^s-p^{s-k}+ip^{s-k-1} \\ & \text{ with } 1\leq i \leq p-1 \text{ and } 1 \leq k \leq s-1;\\
0  & \text{ if } \upsilon=p^s.
\end{array}\right. \vspace{-3mm}\end{equation*}
\end{thm}
 \begin{proof} Working in a similar manner as in Theorem 3.4 of Dinh \cite{dinh4}, the desired result follows.\vspace{-2mm}\end{proof}
From now onwards, throughout this paper, we shall follow the same notations as in Section \ref{prelim},  and    we shall focus our attention on constacyclic codes of length $N=np^s$ over $\mathcal{R}_{e},$ where $p$ is a prime and $s,n$ are positive integers with $\gcd(n,p)=1.$ 
\vspace{-6mm}\section{Repeated-root constacyclic codes over $\mathcal{R}_{e}$ when $e=2$ and their distances}\label{sec3}\vspace{-3mm}
In  this section, we shall assume that $e=2,$ and  we shall determine  all repeated-root constacyclic codes of length $np^s$ over $\mathcal{R}_2$ and their dual codes, where $\mathcal{R}_2$ is a finite commutative chain ring with nilpotency index 2.  We shall also determine the number of codewords in each code and list some isodual constacyclic codes of length $np^s$ over $\mathcal{R}_{2}.$  We shall also determine their Hamming distances, RT distances and RT weight distributions. 

To do this,   by Proposition \ref{teich}(a),  we see  that the characteristic of $\mathcal{R}_2$ is either $p$ or $p^2.$ Moreover, by Proposition \ref{teich}(b), we see that there exists  an element $\zeta \in \mathcal{R}_2$ having the multiplicative order as  $p^{m}-1$ and that $\mathcal{R}_2=\{\alpha+\gamma \beta: \alpha, \beta \in \mathcal{T}_{2}\},$ where $\mathcal{T}_{2}= \{ 0,1, \zeta,\cdots,\zeta^{p^{m}-2} \}$  is the Teichm\"{u}ller set of $\mathcal{R}_2.$   We also recall that for a unit $\lambda \in \mathcal{R}_2,$ a $\lambda$-constacyclic code of length $np^s$ over $\mathcal{R}_2$ is an ideal of the quotient ring $\mathcal{R}_2[x]/\langle x^{np^s}-\lambda\rangle.$ By Proposition \ref{teich}(c),  the unit $\lambda\in \mathcal{R}_{2}$ can be uniquely written as $\lambda=\alpha + \gamma \beta,$ where $\alpha, \beta \in \mathcal{T}_{2}$ and $\alpha \neq 0.$ Further, by Proposition \ref{teich}(c) again, we see that there exists $\alpha_0 \in \mathcal{T}_{2}$ satisfying ${\alpha_0}^{p^s}=\alpha.$  This implies that $x^{np^s}-\lambda= x^{np^s}-{\alpha_0}^{p^s}-\gamma \beta.$ Now by Lemma \ref{coprime}(b), we can write $x^n-\alpha_0=f_1(x)f_2(x)\cdots f_r(x),$ where  $f_1(x),f_2(x),\cdots, f_r(x)$ are monic basic irreducible pairwise coprime polynomials in $\mathcal{R}_2[x].$ Further, by applying Lemma \ref{coprime}(a), we observe  that the polynomials $ f_1(x)^{p^s},f_2(x)^{p^s},\cdots, f_r(x)^{p^s}$ are pairwise coprime in $\mathcal{R}_2[x]$  and that  the polynomials $f_j(x)$ and $F_j(x)=\frac{x^n-\alpha_0}{f_j(x)}$ are coprime in $\mathcal{R}_2[x]$ for $1 \leq j \leq r.$  Moreover, for $1 \leq u \leq r-1,$ by Lemma \ref{coprime}(a) again, we see that the polynomials $f_u(x)^{p^s}$ and $f_{u+1}(x)^{p^s}f_{u+2}(x)^{p^s}\cdots f_r(x)^{p^s}$ are coprime in $\mathcal{R}_2[x],$  which implies that there exist $v_u(x), w_u(x) \in \mathcal{R}_2[x]$ satisfying $\text{deg }w_u(x) < \text{deg }f_u(x)^{p^s}$ and  $v_u(x)f_u(x)^{p^s}+w_u(x) f_{u+1}(x)^{p^s}f_{u+2}(x)^{p^s}\cdots f_r(x)^{p^s}=1.$  Moreover, by Proposition \ref{bino}, we see that $p ||  \binom{p^s}{kp^{s-1}}$ for $1 \leq k \leq p-1$ and that $p^2 | \binom{p^s}{i}$ for $1 \leq i \leq p^s-1$ satisfying $p^{s-1} \nmid i.$ Thus when $\text{char }\mathcal{R}_{2}=p^2,$ we can write $ \binom{p^s}{kp^{s-1}}=pa_k$ with $p \nmid a_k$ for $1\leq k \leq p-1$ and that $ \binom{p^s}{i} =0$ in $\mathcal{R}_2$ for each $i$ satisfying $1 \leq i \leq p^s-1$ and $ p^{s-1} \nmid i.$  Further, when $\text{char } \mathcal{R}_2=p^2,$ we see that $p(=p.1)=\gamma z$ in $\mathcal{R}_2,$ where $z \in \mathcal{T}_{2}\setminus \{0\}.$  In view of this, we factorize the polynomial $x^{np^s}-\lambda$ into pairwise coprime polynomials in $\mathcal{R}_2[x]$ in the following lemma. 
\vspace{-2mm} \begin{lem} \label{fac}We have $x^{np^s}-\lambda= \prod\limits_{j=1}^{r} \left(f_j(x)^{p^s}+\gamma g_j(x)\right),$ where the polynomials  $g_1(x), g_2(x),\cdots,g_r(x) \in \mathcal{R}_2[x]$ satisfy the following for $1 \leq j \leq r:$
  \begin{itemize} \vspace{-2mm}\item $f_j(x)$ and $g_j(x)$ are coprime in $\mathcal{R}_2[x]$  when $\beta \neq 0.$  \vspace{-2mm}\item $g_j(x)=0$ when $\beta =0 $ and $\text{char }\mathcal{R}_2=p.$ \vspace{-2mm}\item$g_j(x)=f_j(x)^{p^{s-1}}M_j(x)$ when  $\beta =0 $ and $\text{char }\mathcal{R}_2=p^2,$ where \vspace{-3mm}\begin{equation*}M_j(x)=F_j(x){^{p^{s-1}}} \big(z\sum\limits_{k=1}^{p-1} a_k(x^n-\alpha_0)^{(k-1)p^{s-1}}\alpha_0^{p^s-kp^{s-1}} \big) w_{j}(x) \prod\limits_{i=1}^{j-1} v_{i}(x)\vspace{-4mm}\end{equation*} is coprime to $f_j(x) $ in $\mathcal{R}_2[x].$ \vspace{-2mm}\end{itemize}
  Moreover, the polynomials $f_1(x)^{p^s}+\gamma g_1(x), f_2(x)^{p^s}+\gamma g_2(x),\cdots,f_r(x)^{p^s}+\gamma g_r(x)$ are pairwise coprime in $\mathcal{R}_2[x].$
  \vspace{-7mm}\end{lem}
\begin{proof} To prove the result,  we shall distinguish the following two cases: (i) $\text{char }\mathcal{R}_2=p$ and (ii) $\text{char }\mathcal{R}_2=p^2.$
\begin{description}\vspace{-2mm}\item[(i)] First suppose that $\text{char } \mathcal{R}_2=p.$ Here by Proposition \ref{bino}, we see that $\binom{p^s}{i}=0$ in $\mathcal{R}_2$ for $1 \leq i \leq p^s-1.$ Next we observe that $x^{np^s}-\lambda=x^{np^s}-\alpha-\gamma\beta=(x^n-\alpha_0)^{p^s}-\gamma  \beta=f_1(x)^{p^s}f_2(x)^{p^s}\cdots f_r(x)^{p^s}-\gamma \beta .$ As $v_1(x) f_1(x)^{p^s}+w_1(x) f_2(x)^{p^s}f_3(x)^{p^s}\cdots f_r(x)^{p^s}=1,$ we can write \vspace{-2mm}\begin{equation*}x^{np^s}-\lambda=\left\{f_1(x)^{p^s}-\gamma \beta w_1(x)  \right\}  \left\{f_2(x)^{p^s}f_3(x)^{p^s}\cdots f_r(x)^{p^s}-\gamma \beta v_1(x) \right\}.\vspace{-2mm}\end{equation*}  Further, since $v_2(x) f_2(x)^{p^s}+w_2(x) f_3(x)^{p^s}f_4(x)^{p^s}\cdots f_r(x)^{p^s}=1,$ we get $f_2(x)^{p^s}f_3(x)^{p^s}\cdots f_r(x)^{p^s}  -\gamma \beta v_1(x)  =\left\{f_2(x)^{p^s}-\gamma \beta v_1(x)w_2(x) \right\} \left\{ f_3(x)^{p^s}f_4(x)^{p^s}\cdots f_r(x)^{p^s} -\gamma \beta v_1(x)v_2(x)  \right\}.$ Proceeding like this, we see that \vspace{-3mm}\begin{equation*}x^{np^s}-\lambda=\prod\limits_{j=1}^{r}\Big(f_j(x)^{p^s}+\gamma g_j(x)\Big),\vspace{-5mm}\end{equation*} where  $g_1(x)=-\beta$ when $r=1;$ and  $g_{j}(x)=-\beta w_{j}(x) \prod\limits_{i=1}^{j-1} v_{i}(x)$ for $1 \leq j \leq r-1$ and $g_{r}(x)=-\beta \prod\limits_{i=1}^{r-1} v_{i}(x)$ when $r \geq 2.$

\vspace{-2mm}\item[(ii)] Next suppose that  $\text{char }\mathcal{R}_2=p^2.$ Here we have  $ \binom{p^s}{kp^{s-1}}=pa_k$ with $p \nmid a_k$ for $1\leq k \leq p-1,$ $ \binom{p^s}{i} =0$ in $\mathcal{R}_2$ for each $i$ satisfying $1 \leq i \leq p^s-1$ and $ p^{s-1} \nmid i,$ and  $p=\gamma z,$  where $z \in \mathcal{T}_{2}\setminus \{0\}.$
 Using this, we see that  $x^{np^s}-\lambda=(x^n-\alpha_0)^{p^s} +\sum\limits_{k=1}^{p-1} \binom{p^s}{kp^{s-1}}(x^n-\alpha_0)^{kp^{s-1}}{\alpha_0}^{p^s-kp^{s-1}} -\gamma  \beta=f_1(x)^{p^s}f_2(x)^{p^s}\cdots f_r(x)^{p^s}-\gamma \big( \beta -z\sum\limits_{k=1}^{p-1} a_k(x^n-\alpha_0)^{kp^{s-1}}{\alpha_0}^{p^s-kp^{s-1}} \big) .$ Now working in a similar manner as in case (i), we get \vspace{-3mm}\begin{equation*}x^{np^s}-\alpha-\gamma \beta =\prod\limits_{j=1}^{r}\big(f_j(x)^{p^s}+\gamma g_j(x)\big),\vspace{-3mm}\end{equation*} where  $g_1(x)=-\big( \beta -z\sum\limits_{k=1}^{p-1} a_k(x^n-\alpha_0)^{kp^{s-1}}{\alpha_0}^{p^s-kp^{s-1}} \big)$ when $r=1;$ and  $g_{j}(x)=-\big( \beta -z\sum\limits_{k=1}^{p-1} a_k(x^n-\alpha_0)^{kp^{s-1}}{\alpha_0}^{p^s-kp^{s-1}} \big) w_{j}(x) \prod\limits_{i=1}^{j-1} v_{i}(x)$ for $1 \leq j \leq r-1$ and $g_{r}(x)=-\big( \beta -z\sum\limits_{k=1}^{p-1} a_k(x^n-\alpha_0)^{kp^{s-1}}{\alpha_0}^{p^s-kp^{s-1}} \big)\\ \prod\limits_{i=1}^{r-1} v_{i}(x)$  when $r \geq 2.$\vspace{-2mm}\end{description}
From this and by applying Lemma \ref{coprime}(a), the desired result follows immediately.\vspace{-2mm}\end{proof} 
Next for $1 \leq j \leq r,$ let us define  $k_j(x)=f_j(x)^{p^s}+\gamma g_j(x).$  Here we observe that if $\text{deg } f_j(x)=d_j,$ then $\text{deg }k_j(x)=d_jp^s$ for each $j.$ By Lemma \ref{fac}, we see that  $x^{np^s}-\lambda=x^{np^s}-\alpha-\gamma \beta =\prod\limits_{j=1}^{r}k_j(x)$ is a factorization of  $x^{np^s}-\lambda$ into monic pairwise coprime polynomials in $\mathcal{R}_2[x].$ Now by applying the Chinese Remainder Theorem, we get
 \vspace{-3mm}\begin{equation*}\label{decomk} \mathcal{R}_{\alpha,\beta}=\mathcal{R}_2[x]/\langle x^{np^s}-\alpha - \gamma \beta \rangle\simeq \bigoplus\limits_{j=1}^{r} \mathcal{R}_2[x]/\langle k_j(x)\rangle.\vspace{-2mm}\end{equation*} From this point on, let $\mathcal{K}_{j}=\mathcal{R}_2[x]/\langle k_j(x)\rangle$ for  $1 \leq j \leq r.$ Then we have the following:
\vspace{-2mm}\begin{prop}\label{p1}\begin{enumerate}\item[(a)] If $\mathcal{C}$ is an $(\alpha+\gamma\beta)$-constacyclic code of length $np^s$ over $\mathcal{R}_2$ (i.e.,  an ideal of the ring $\mathcal{R}_{\alpha,\beta}$), then $\mathcal{C}=\mathcal{C}_1\oplus \mathcal{C}_2\oplus \cdots \oplus \mathcal{C}_r,$ where $\mathcal{C}_j$ is an ideal of $\mathcal{K}_{j}$ for $1 \leq j \leq r.$
\vspace{-2mm}\item[(b)] Let $I_j$ be an ideal of $\mathcal{K}_{j}$ for $1 \leq j \leq r.$ Then $I=I_1 \oplus I_2\oplus \cdots \oplus I_r$ is an ideal of $\mathcal{R}_{\alpha,\beta},$ (i.e., $I$ is an $(\alpha+\gamma \beta )$-constacyclic code of length $np^s$ over $\mathcal{R}_2$).  Furthermore, we have $|I|=|I_1||I_2|\cdots |I_r|.$\end{enumerate}\vspace{-2mm}\end{prop}
\noindent\textbf{Proof.} Its proof is straightforward.  $\hfill\Box$

Now let $\mathcal{C}$ be an $(\alpha+\gamma\beta)$-constacyclic code of length $np^s$ over $\mathcal{R}_2.$ Then its dual code $\mathcal{C}^{\perp}$ is an $(\alpha+\gamma \beta)^{-1}$-constacyclic code of length $np^s$ over $\mathcal{R}_2.$  Further, we see that $(\alpha+\gamma \beta)^{-1}=\alpha^{-1}-\gamma \beta \alpha^{-2},$ which implies that $\mathcal{C}^{\perp}$ is an ideal of the ring $\widehat{\mathcal{R}_{\alpha,\beta}}=\mathcal{R}_2[x]/\langle x^{np^s}-(\alpha + \gamma \beta )^{-1}  \rangle= \mathcal{R}_{\alpha^{-1},-\beta\alpha^{-2}}.$ To determine $\mathcal{C}^{\perp},$ we see that $x^{np^s}-(\alpha+\gamma \beta )^{-1}=-(\alpha+\gamma \beta )^{-1}k_1^*(x)k_2^*(x)\cdots k_r^*(x).$  By applying Chinese Remainder Theorem again, we obtain $\widehat{\mathcal{R}_{\alpha,\beta}}\simeq \bigoplus\limits_{j=1}^{r}\widehat{\mathcal{K}_{j}},$ where $\widehat{\mathcal{K}_{j}}=\mathcal{R}_2[x]/\langle k_j^*(x)\rangle$ for $1 \leq j \leq r.$ Now we make the following observation.
\vspace{-1mm}\begin{prop}\label{p2} Let $\mathcal{C}$ be an $(\alpha+\gamma \beta)$-constacyclic code of length $np^s$ over $\mathcal{R}_2,$ i.e.,  an ideal of the ring $\mathcal{R}_{\alpha,\beta}.$ If $\mathcal{C}=\mathcal{C}_1\oplus \mathcal{C}_2\oplus \cdots \oplus \mathcal{C}_r$ with $\mathcal{C}_j$ an ideal of the quotient ring $\mathcal{K}_{j}$ for each $j,$  then the dual code $\mathcal{C}^{\perp}$ of $\mathcal{C}$ is given by $\mathcal{C}^{\perp}=\mathcal{C}_1^{\perp}\oplus \mathcal{C}_2^{\perp}\oplus \cdots \oplus \mathcal{C}_r^{\perp},$ where $\mathcal{C}_j^{\perp} =\{a_j(x) \in \widehat{\mathcal{K}_{j}}: a_j(x) c_j^*(x)=0 \text{ in }\widehat{\mathcal{K}_{j}} \text{ for all }c_j(x) \in \mathcal{C}_{j}\}$ is the orthogonal complement of  $\mathcal{C}_j$  for each $j.$ Moreover, $\mathcal{C}_j^{\perp}$ is an ideal of the quotient ring $\widehat{\mathcal{K}_{j}}=\mathcal{R}_2[x]/\langle k_j^*(x)\rangle$ for each $j.$\vspace{-2mm}\end{prop}
\vspace{-1mm}\noindent\textbf{Proof.}   Proof is trivial. $\hfill \Box$

In view of Propositions \ref{p1} and \ref{p2}, we see that to determine all $(\alpha+\gamma\beta)$-constacyclic codes of length $np^s$ over $\mathcal{R}_2,$ their sizes  and their dual codes, we need to determine all ideals of the quotient ring $\mathcal{K}_{j},$ their cardinalities and their orthogonal complements in the quotient ring $\widehat{\mathcal{K}_{j}}$ for $1 \leq j \leq r.$  For this, throughout this paper, let $1 \leq j \leq r$ be fixed. From now on, we shall represent elements of the quotient rings $\mathcal{K}_{j}$ and $\widehat{\mathcal{K}_{j}}$ (resp. $\mathcal{R}_2[x]/\langle f_j(x)^{p^s}\rangle$ and $\overline{\mathcal{R}_2}[x]/\langle\overline{ f_j(x)}^{p^s}\rangle$) by their representatives in $\mathcal{R}_2[x]$ (resp. $\overline{\mathcal{R}_2}[x]$) of degree less than $d_jp^s$ (resp. $d_jp^s$), and we shall perform their addition and multiplication modulo $k_j(x)$ and $k_j^*(x)$ (resp.  $f_j(x)^{p^s}$ and $\overline{f_j(x)}^{p^s}$), respectively. Now to determine all ideals of the quotient ring $\mathcal{K}_{j},$ their orthogonal complements and their sizes,  we  shall first prove the following lemma:
\vspace{-2mm}\begin{lem} \label{nilred} Let $1 \leq j \leq r$ be fixed. In the ring $\mathcal{K}_{j},$ we have the following:\begin{enumerate}\vspace{-2mm}\item[(a)] Any non-zero polynomial $a(x) \in \mathcal{R}_2[x]$ that is coprime to $f_j(x)$ is a unit in $\mathcal{K}_{j}.$ As a consequence, any non-zero polynomial $b(x)\in \mathcal{R}_2[x]$ satisfying $\text{deg }b(x) < d_j$ and  $\overline{ b(x)} \neq 0$ is a unit in $\mathcal{K}_{j}.$
\vspace{-2mm}\item[(b)]$f_j(x)$ is nilpotent in $\mathcal{K}_{j}$ and  $\langle f_j(x)^{p^s}\rangle=\left\{\begin{array}{ll}\langle\gamma \rangle & \text{if }\beta \neq 0;\\ \{0 \} & \text{if }\beta=0 \text{ and } \text{char }\mathcal{R}_2=p; \\ \langle\gamma f_j(x)^{p^{s-1}} \rangle  & \text{if } \beta=0 \text{ and } \text{char }\mathcal{R}_2=p^2.\end{array}\right.$ 
\vspace{-2mm}\item [(c)] The nilpotency index $\mathfrak{N}$ of $f_j(x)$ is given by \vspace{-4mm}$$\mathfrak{N}= \left\{\begin{array}{ll} 2p^s & \text{if }\beta \neq 0;\\ p^s & \text{if }\beta=0 \text{ and } \text{char }\mathcal{R}_2=p; \\ 2p^s-p^{s-1}  & \text{if } \beta=0 \text{ and } \text{char }\mathcal{R}_2=p^2.\end{array}\right.$$
\vspace{-2mm}\end{enumerate}\vspace{-9mm}\end{lem}
 \vspace{-2mm}\begin{proof}\begin{enumerate}\item[{\it (a)}] As $a(x)\in \mathcal{R}_2[x]$ is coprime to $f_j(x)$ and $f_j(x)$ is a basic irreducible polynomial in $\mathcal{R}_2[x],$ by Lemma \ref{coprime}(a), we see that the polynomials  $f_j(x)^{p^s}$ and $a(x)$ are coprime in $\mathcal{R}_2[x],$ which implies that there exist polynomials $q(x),r(x) \in \mathcal{R}_2[x]$ such that $q(x)a(x)+r(x)f_j(x)^{p^s}=1$ in $\mathcal{R}_2[x].$ This gives $q(x)a(x)=1+\gamma r(x)g_j(x)$ in $\mathcal{K}_{j}.$ From this and  using the fact that  $\gamma^2=0$ in $\mathcal{K}_{j},$ we see that  $a(x)$ is a unit in $\mathcal{K}_{j}.$

On the other hand, if  $b(x)\in \mathcal{R}_2[x]$ satisfies $\text{deg }b(x) < d_j$ and  $\overline{ b(x)} \neq 0,$ then by applying Lemma \ref{coprime}(a), we see that $b(x)$ and $f_j(x)$ are coprime in $\mathcal{R}_2[x],$ from which the desired result follows.
\vspace{-2mm}\item[{\it (b)}] In $\mathcal{K}_{j},$ we see that $f_j(x)^{p^s}=-\gamma g_j(x) \in \langle \gamma \rangle,$ which implies that $f_j(x)$ is nilpotent in $\mathcal{K}_j.$
 \\When $\beta \neq 0,$ by Lemma \ref{fac}, we see that $f_j(x)$ and $g_j(x)$ are coprime in $\mathcal{R}_2[x].$ Now by part (a), we note that $g_j(x)$ is a unit in $\mathcal{K}_{j},$ which implies that $\langle f_j(x)^{p^s}\rangle=\langle \gamma \rangle.$
Next when $\beta=0$ and $\text{char }\mathcal{R}_2=p,$ by Lemma \ref{fac}, we have $g_j(x)=0,$ which gives $f_j(x)^{p^s}=0$ in $\mathcal{K}_{j}.$
Finally, when $\beta=0$ and $\text{char }\mathcal{R}_2=p^2,$ by Lemma \ref{fac}, we have $g_j(x)= f_j(x)^{p^{s-1}}M_j(x),$ where $M_j(x)$ is coprime to $f_j(x)$ in $\mathcal{R}_2[x].$ This implies that $f_j(x)^{p^s}=-\gamma f_j(x)^{p^{s-1}}M_j(x).$ By part  (a), we see that $M_j(x)$ is a unit in $\mathcal{K}_j.$ From this, we obtain $\langle f_j(x)^{p^s} \rangle =\langle\gamma f_j(x)^{p^{s-1}} \rangle.$
\vspace{-2mm}\item[{\it (c)}]  When $\beta \neq 0,$ by part (b), we have $\langle f_j(x)^{p^s}\rangle=\langle\gamma \rangle,$ which implies that the nilpotency index of $f_j(x)$ in $\mathcal{K}_j$ is $2p^s.$
  Next when $\beta=0$ and $\text{char }\mathcal{R}_2=p,$ by part (b), we have $\langle f_j(x)^{p^s}\rangle =\{0 \}. $ From this and using the fact that $f_j(x)^{p^{s}-1} \neq 0$ in $\mathcal{K}_j,$ we see that the nilpotency index of $f_j(x)$ in $\mathcal{K}_j$ is $p^s.$ Finally, when $\beta=0$ and $\text{char }\mathcal{R}_2=p^2,$ by part (b), we see that $\langle f_j(x)^{p^s}\rangle = \langle\gamma f_j(x)^{p^{s-1}} \rangle.$ This implies that  $\gamma f_j(x)^{p^s}=0.$  We further observe that $\gamma f_j(x)^{p^s-1}\neq 0$ in $\mathcal{K}_{j}.$ From this, it follows that the nilpotency index of $f_j(x)$ in $\mathcal{K}_j$ is $2p^s-p^{s-1}.$
\vspace{-1mm}\end{enumerate}\vspace{-4mm}\end{proof}
\vspace{-2mm}For a positive integer $k$ and a subset  $\mathcal{S}$ of $\mathcal{R}_2$ with $0 \in \mathcal{S},$  let us define $\mathcal{P}_{k}(\mathcal{S})=\{g(x) \in \mathcal{S}[x] : \text{ either }g(x)=0 \text{ or }\text{deg }g(x) < k\}.$ By repeatedly applying division algorithm in $\mathcal{R}_2[x],$ every element $A(x) \in \mathcal{K}_j$ can be  uniquely written as $A(x)= \sum\limits_{i=0}^{p^s-1}  A_i(x)f_j(x)^i,$ where $A_i(x) \in \mathcal{P}_{d_j}(\mathcal{R}_2)$ for $0 \leq i \leq p^s-1.$ Further, each $A_i(x) \in \mathcal{P}_{d_j}(\mathcal{R}_2)$
 can be uniquely expressed as $A_i(x)=A_{i0}(x)+\gamma A_{i1}(x),$  where $A_{i0}(x), A_{i1}(x) \in \mathcal{P}_{d_j}(\mathcal{T}_{2}).$ In view of this, we see that every element $A(x) \in \mathcal{K}_{j}$ can be uniquely expressed as $A(x)= \sum\limits_{i=0}^{p^s-1}  A_{i0}(x)f_j(x)^i+ \gamma \sum\limits_{i=0}^{p^s-1}  A_{i1}(x)f_j(x)^i,$ where $A_{i0}(x), A_{i1}(x) \in \mathcal{P}_{d_j}(\mathcal{T}_{2})$ for each $i.$  
 
The following lemma  is useful to determine cardinalities of ideals of $\mathcal{K}_{j}.$

\vspace{-2mm}\begin{lem} \label{card} Let $1 \leq j \leq r$ be fixed and let $\mathcal{I}$ be an ideal of $\mathcal{K}_j.$ Then $\text{Res}_{\gamma}(\mathcal{I})=\Big\{ \overline{a_0(x)} \in \overline{\mathcal{R}_2}[x]/\langle\overline{f_j(x)}^{p^s}\rangle: a_0(x)+\gamma a_1(x) \in \mathcal{I} \text{ for some }a_0(x),a_1(x) \in \mathcal{T}_{2}[x]\Big\}$ and $Tor_{\gamma}(\mathcal{I})=\Big\{\overline{a_1(x)} \in \overline{\mathcal{R}_2}[x]/\langle\overline{f_j(x)}^{p^s}\rangle: \gamma a_1(x)\in \mathcal{I} \text{ for some } a_1(x) \in \mathcal{T}_{2}[x]\Big\}$ are ideals of $\overline{\mathcal{R}_2}[x]/\langle\overline{f_j(x)}^{p^s}\rangle.$ Moreover, we have $|\mathcal{I}|=|Res_{\gamma}(\mathcal{I})||Tor_{\gamma}(\mathcal{I})|.$\vspace{-3mm} \end{lem}
\vspace{-3mm}\begin{proof} One can easily show that $\text{Res}_{\gamma}(\mathcal{I})$ and $\text{Tor}_{\gamma}(\mathcal{I})$ are ideals of 
$\overline{\mathcal{R}_2}[x]/\langle\overline{f_j(x)}^{p^s}\rangle.$ To prove the second part, we shall view $\mathcal{R}_2$ as an $\overline{\mathcal{R}_2}-$module with respect to the addition in $\mathcal{R}_2$ and the scalar product defined as $\overline{a}r= a r $ for each $a \in \mathcal{T}_{2}$ and $r \in \mathcal{R}_2.$  Further,  we note that  $\overline{\mathcal{R}_2}[x]/\langle\overline{f_j(x)}^{p^s}\rangle$ can be viewed as an $\overline{\mathcal{R}_2}$-module. Thus the ideals  $\text{Res}_{\gamma}(\mathcal{I})$ and $\text{Tor}_{\gamma}(\mathcal{I})$ of $\overline{\mathcal{R}_2}[x]/\langle\overline{f_j(x)}^{p^s}\rangle$ can also be viewed as  $\overline{ \mathcal{R}_2}$-modules. Now define a map $\phi : \mathcal{I} \rightarrow \text{Res}_{\gamma}(\mathcal{I})$ as $\phi(a(x))=\overline{ a_0(x)}$ for each $a(x)=a_0(x)+\gamma  a_1(x)   \in\mathcal{I} $ with $a_0(x),a_1(x)\in \mathcal{T}_{2}[x].$ We see that $\phi$ is a surjective $\overline{ \mathcal{R}_2}$-module homomorphism and its  kernel is given by $\text{ker } \phi =\{a_0(x)+\gamma a_1(x) \in \mathcal{I}:\overline{ a_0(x)}=0\} = \{\gamma a_1(x) \in \mathcal{I}: a_1(x) \in \mathcal{T}_{2}[x]\}.$  From this, we get $|\mathcal{I}|=|\text{Res}_{\gamma}(\mathcal{I})| |\text{ker } \phi|.$ Further, one can easily see that  $|\text{Tor}_{\gamma}(\mathcal{I})| =|\text{ker } \phi|,$ from which the desired result follows immediately.
\end{proof}
\vspace{-1mm}The following lemma is useful to determine  orthogonal complements of all ideals of the ring $\mathcal{K}_{j}.$
\vspace{-2mm}\begin{lem} \label{lemdual} Let $1 \leq j \leq r $ be a fixed integer. Let $\mathcal{I}$ be an ideal of the ring $\mathcal{K}_{j},$ and let $\mathcal{I}^{\perp}$ be  the orthogonal complement of $\mathcal{I}$ in $\widehat{\mathcal{K}_{j}}.$  Then the following hold.\vspace{-1mm}
\begin{enumerate}
\vspace{-1mm}\item[(a)] $\mathcal{I}^{\perp}$ is an ideal of $\widehat{\mathcal{K}_{j}}.$
\vspace{-2mm}\item[(b)] $\mathcal{I}^{\perp}=\{a^{*}(x)\in \widehat{\mathcal{K}_{j}}: a(x) \in \text{ann}(\mathcal{I})\}=\text{ann}(\mathcal{I})^*.$
\vspace{-2mm}\item[(c)] If $\mathcal{I}=\langle f(x),\gamma g(x)\rangle,$ then $\mathcal{I}^{*}=\langle f^*(x),\gamma g^*(x)\rangle.$
\vspace{-2mm}\item[(d)] For $f(x),g(x) \in \mathcal{K}_{j},$ let us define $(fg)(x)=f(x)g(x)$ and $(f+g)(x)=f(x)+g(x).$\\
 If $f(x), g(x),(fg)(x)$ all are non-zero, then we have  $f^*(x) g^*(x)=x^{\text{deg }f(x)+\text{deg }g(x)-\text{deg }(fg)(x)}(fg)^*(x).$\\ If $f(x), g(x),(f+g)(x)$ all are non-zero, then we have
 \vspace{-3mm}$$(f+g)^*(x)=\left\{\begin{array}{ll} 
f^*(x)+x^{deg\, f(x)-deg\, g(x)}g^*(x)  & \text{if } \text{deg }f(x) > \text{deg } g(x);\\
x^{\text{deg }(f+g)(x)-\text{deg f(x)}}(f^*(x)+g^*(x)) & \text{if } \text{deg } f(x) = \text{deg }g(x).\end{array}\right.$$\vspace{-4mm}
\end{enumerate}\vspace{-7mm}\end{lem}\begin{proof} Its proof is straightforward.\vspace{-3mm}\end{proof}

From now on,  we shall distinguish the following two cases:  $\beta \neq 0$ and $\beta =0.$ 

In the following theorem,  we determine all ideals of the ring $\mathcal{K}_j,$ their sizes and their orthogonal complements in $\widehat{\mathcal{K}_{j}}$ when $\beta$ is non-zero.
\vspace{-2mm}\begin{thm}  \label{thm1} When $\beta \neq 0,$ the ring $\mathcal{K}_j$ is a finite commutative chain ring with unity whose ideals are given by $\{ 0 \} \subset \langle f_j(x)^{2p^s-1}\rangle \subset \langle f_j(x)^{2p^s-2}\rangle \subset \dots\subset \langle f_j(x)^{2}\rangle \subset \langle f_j(x)\rangle \subset \mathcal{K}_j.$ Moreover, for $0 \leq \nu \leq 2p^s,$ the ideal $\langle f_j(x)^\nu \rangle$ has  $ p^{m d_j (2p^s-\nu)}$ elements and the orthogonal complement of $\langle f_j(x)^{\nu}\rangle$ is given by $\langle f_j^*(x)^{2p^s-\nu}\rangle.$
\vspace{-2mm}\end{thm}
\vspace{-4mm}\begin{proof} To prove this, we see that each element $A(x)\in \mathcal{K}_{j}$ can be  uniquely expressed as $A(x)= \sum\limits_{i=0}^{p^s-1}  A_{i0}(x)f_j(x)^i+ \gamma \sum\limits_{i=0}^{p^s-1}  A_{i1}(x)f_j(x)^i,$ where $A_{i0}(x), A_{i1}(x) \in \mathcal{P}_{d_j}(\mathcal{T}_{2})$ for each $i.$ As $f_j(x)$ and $\gamma $  are nilpotent  in $\mathcal{K}_j,$ we see that $A(x)$ is a unit in $\mathcal{K}_{j}$ if and only if $A_{00}(x)$ is a unit in $\mathcal{K}_{j}.$ Further, by Lemma \ref{nilred}(a), we observe that $A_{00}(x) \in \mathcal{P}_{d_j}(\mathcal{T}_{2})$ is a unit in $\mathcal{K}_{j}$ if and only if $A_{00}(x) \neq 0.$  In view of this and by applying Lemma \ref{nilred}(b), we see that  $A(x) $ is a unit in $\mathcal{K}_{j}$ if and only if $A(x) \notin \langle f_j(x)\rangle.$ This shows that all the non-units of $\mathcal{K}_{j}$ are given by  $\langle f_j(x) \rangle.$ Therefore $\mathcal{K}_{j}$ is a local ring with the unique maximal ideal  as $\langle f_j(x) \rangle.$ This, by Proposition \ref{pr1} and Lemma \ref{nilred}(b), implies that $\mathcal{K}_{j}$ is a chain ring and all its ideals  are given by $\langle f_j(x) ^{\nu}\rangle,$ where  $0 \leq \nu \leq 2p^s.$ Further, we observe that $|\overline{\mathcal{K}_j}|=|\mathcal{K}_j/\langle f_j(x)\rangle|=p^{md_j}.$  Now by applying Proposition \ref{pr1} and Lemma \ref{nilred}(b) again, we see that $| \langle f_j(x) ^{\nu}\rangle|=p^{md_j(2p^s-\nu)}$ for $0 \leq \nu \leq 2p^s.$  In order to determine their dual codes, let $\mathcal{I}=\langle f_j(x)^{\nu} \rangle,$ where $0 \leq \nu \leq 2p^s.$  Here it is easy to observe that  $\text{ann}(\mathcal{I})=\langle f_j(x)^{2p^s-\nu}\rangle,$ which, by Lemma \ref{lemdual},  gives $\mathcal{I}^{\perp}=\text{ann}(\mathcal{I})^*=\langle f_j^*(x)^{2p^s-\nu}\rangle.$ This completes the proof of the theorem.
\vspace{-2mm}\end{proof}
As a consequence of the above theorem, we deduce the following:
\vspace{-2mm}\begin{cor}\label{c1} Let $\alpha =\alpha_0^{p^s}\in \mathcal{T}_{2} \setminus \{0\},$ where $\alpha_0 \in \mathcal{T}_{2}$ is such that $x^n-\alpha_0$ is  irreducible  over $\mathcal{R}_2 .$ When $\beta \in \mathcal{T}_{2}\setminus \{0\},$   $\langle (x^n-\alpha_0)^{p^{s}}\rangle=\langle\gamma \rangle$ is the only isodual $(\alpha+\gamma \beta)$-constacyclic code of length $np^s$ over $\mathcal{R}_2.$\vspace{-2mm}\end{cor}
\begin{proof}  As $x^n-\alpha_0$  is irreducible over $\mathcal{R}_2,$ by 
Theorem \ref{thm1},  we see that  all  $(\alpha+\gamma \beta)$-constacyclic codes of length $np^s$ over $\mathcal{R}_2$ are given by $\langle (x^n-\alpha_0)^{\nu}\rangle,$ where $0 \leq \nu \leq 2p^s.$  If $\mathcal{I}=\langle (x^n-\alpha_0)^{\nu}\rangle,$ then  by Theorem \ref{thm1} again,  we note that $|\mathcal{I}|=p^{m n(2p^s-\nu)}$ and $\mathcal{I}^{\perp}=\langle (x^n-\alpha_0^{-1})^{2p^s-\nu}\rangle$ for each $\nu.$  Working as in Theorem \ref{thm1}, we see that $|\mathcal{I}^{\perp}|= p^{mn\nu}$ for $0 \leq \nu \leq 2p^s.$
Now if the code $\mathcal{I}=\langle (x^n-\alpha_0)^{\nu}\rangle$ is isodual, then we must have $|\mathcal{I}|=|\mathcal{I}^{\perp}|,$ which implies that $\nu=p^s.$ On the other hand, we see that the codes  $\langle (x^n-\alpha_0)^{p^{s}}\rangle=\langle\gamma \rangle (\subseteq \mathcal{R}_{\alpha,\beta})$ and $\langle(x^n-\alpha_0^{-1})^{p^{s}} \rangle =\langle \gamma \rangle (\subseteq \widehat{\mathcal{R}_{\alpha,\beta}})$ are clearly $\mathcal{R}_2$-linearly equivalent, which completes the proof.\vspace{-1mm}
 \end{proof}

In the following theorem, we detemine Hamming distances of all $(\alpha+\gamma \beta)$-constacyclic codes of length $np^s$ over $\mathcal{R}_2$  when $\beta$ is non-zero and $x^n-\alpha_0$ is irreducible over $\mathcal{R}_2.$

\vspace{-2mm}\begin{thm}\label{dthm1}  Let $\alpha =\alpha_0^{p^s}\in \mathcal{T}_{2} \setminus \{0\},$ where $\alpha_0 \in \mathcal{T}_{2}$ is such that $x^n-\alpha_0$ is  irreducible  over $\mathcal{R}_{2} .$ Let  $\beta \in \mathcal{T}_{2}\setminus \{0\},$ and let  $\mathcal{C}=\langle (x^n-\alpha_0)^{\nu} \rangle$ be an $(\alpha+\gamma \beta)$-constacyclic code of length $np^s$ over $\mathcal{R}_2,$ where $ 0 \leq \nu \leq 2p^s.$ Then the Hamming distance $d_H(\mathcal{C})$ of the code $\mathcal{C}$ is given by
\vspace{-2mm}\begin{equation*}d_H(\mathcal{C})=\left\{\begin{array}{ll}
1 & \text{ if } 0 \leq \nu \leq p^s;\\
\ell+2 & \text{ if } p^s+\ell p^{s-1}+1 \leq \nu \leq p^s+(\ell +1)p^{s-1}, \text{ where } 0 \leq \ell \leq p-2;\\
(i+1)p^k & \text{ if } 2p^s-p^{s-k}+(i-1)p^{s-k-1}+1\leq \nu \leq 2p^s-p^{s-k}+ip^{s-k-1}, \\ & \text{ where } 1\leq i \leq p-1 \text{ and } 1 \leq k \leq s-1;\\
0  & \text{ if } \nu=2p^s.
\end{array}\right. \vspace{-2mm}\end{equation*}
\end{thm}
\begin{proof} It is easy to see that  $d_H(\mathcal{C})=0$ when $\nu=2p^s,$ and that   $d_H(\mathcal{C})=1$  when $\nu=0.$ 
Further, by Lemma \ref{nilred}(c), we see that $\langle (x^n-\alpha_0)^{p^s}\rangle=\langle \gamma \rangle,$ which implies that $\gamma \in \langle (x^n-\alpha_0)^{\nu} \rangle$ for  $1 \leq \nu \leq p^{s}.$ From this, we get $d_H(\mathcal{C})=1$ for $1 \leq \nu \leq p^s.$ 
Next for $p^s+1 \leq \nu \leq 2p^s-1,$ we see that the code $\mathcal{C}=\langle\gamma (x^n-\alpha_0)^{\nu-p^{s}} \rangle.$  From this, we observe that $d_H(\mathcal{C})$ is equal to the Hamming distance of  the $\overline{\alpha_0}^{p^s}$-constacyclic code $\langle (x^n-\overline{\alpha}_0)^{\nu-p^{s}} \rangle$ of length $np^s$ over  $\overline{\mathcal{R}_2}$ for $ p^s+1 \leq \nu \leq 2p^s-1.$ Now using Theorem \ref{dthm}, the desired result follows.\vspace{-2mm}\end{proof}

In the following theorem, we determine RT distances of all $(\alpha+\gamma \beta)$-constacyclic codes of length $np^s$ over $\mathcal{R}_2$ when $\beta \neq 0$ and $x^n-\alpha_0$ is irreducible over $\mathcal{R}_2.$
\vspace{-2mm}\begin{thm}\label{dRthm1} Let $\alpha =\alpha_0^{p^s}\in \mathcal{T}_{2} \setminus \{0\},$ where $\alpha_0 \in \mathcal{T}_{2}$ is such that $x^n-\alpha_0$ is  irreducible  over $\mathcal{R}_{2} .$ Let  $\beta \in \mathcal{T}_{2}\setminus \{0\},$ and let  $\mathcal{C}=\langle (x^n-\alpha_0)^{\nu} \rangle$ be an $(\alpha+\gamma \beta)$-constacyclic code of length $np^s$ over $\mathcal{R}_2,$ where $ 0 \leq \nu \leq 2p^s.$ Then the RT distance $d_{RT}(\mathcal{C})$ of the code $\mathcal{C}$ is given by 
\vspace{-4mm}\begin{equation*}d_{RT}(\mathcal{C})=\left\{\begin{array}{ll}
1 & \text{ if } 0 \leq \nu \leq p^s;\\
n\nu-np^{s}+1  & \text{ if } p^s+1 \leq \nu \leq 2p^s-1;\\
0  & \text{ if } \nu =2p^s.
\end{array}\right. \vspace{-3mm}\end{equation*}
\end{thm}
\begin{proof} When $\nu=2p^s,$ we see that $d_{RT}(\mathcal{C})=0.$
Further, by Lemma \ref{nilred}(b), we have $\langle (x^n-\alpha_0)^{p^s} \rangle=\langle \gamma \rangle,$ which implies that $\gamma \in \langle (x^n-\alpha_0)^{\nu} \rangle $  for $1 \leq \nu \leq p^s.$ This implies that $d_{RT}(\mathcal{C})=1$  for $1 \leq \nu \leq p^s.$

Next for $p^s+1 \leq \nu \leq 2p^s-1,$  we see that  $\mathcal{C}=\langle (x^n-\alpha_0)^{\nu} \rangle=\langle \gamma(x^n-\alpha_0)^{\nu-p^s} \rangle=\{\gamma(x^n-\alpha_0)^{\nu-p^s} f(x) \ | \ f(x) \in \mathcal{T}_{2}[x]\}.$ From this, it follows that $w_{RT}(Q(x)) \geq w_{RT}(\gamma(x^n-\alpha_0)^{\nu-p^s})=  n\nu-np^s+1$ for each $Q(x) \in \mathcal{C} \setminus \{0 \}.$ Moreover, since $w_{RT}((x^n-\alpha_0)^{\nu})=w_{RT}(\gamma (x^n-\alpha_0)^{\nu-p^s})=n \nu-np^s+1,$ we get $d_{RT}(\mathcal{C})=n\nu-np^{s}+1,$ which completes the proof.  \vspace{-2mm}\end{proof}
In the following theorem, we determine RT weight distributions of all $(\alpha+\gamma \beta)$-constacyclic codes of length $np^s$ over $\mathcal{R}_2$ when $\beta \neq 0$ and $x^n-\alpha_0$ is irreducible over $\mathcal{R}_2.$
\vspace{-2mm}\begin{thm} \label{dthmRw} Let $\alpha =\alpha_0^{p^s}\in \mathcal{T}_{2} \setminus \{0\},$ where $\alpha_0 \in \mathcal{T}_{2}$ is such that $x^n-\alpha_0$ is  irreducible  over $\mathcal{R}_{2} .$ Let  $\beta \in \mathcal{T}_{2}\setminus \{0\},$ and let  $\mathcal{C}=\langle (x^n-\alpha_0)^{\nu} \rangle$ be an $(\alpha+\gamma \beta)$-constacyclic code of length $np^s$ over $\mathcal{R}_2,$ where $ 0 \leq \nu \leq 2p^s.$ For $0 \leq \rho \leq np^s,$ let $\mathcal{A}_\rho$ denote the number of codewords in $\mathcal{C}$ having the RT weight as $\rho.$\begin{enumerate}
\vspace{-2mm}\item[(a)] For $\nu=2p^s,$ we have 
\begin{equation*}\mathcal{A}_\rho=\left\{\begin{array}{ll}
1 & \text{ if } \rho=0;\\
0  & \text{ otherwise}.
\end{array}\right. \vspace{-2mm}\end{equation*}
\vspace{-7mm}\item[(b)] For $p^s+1 \leq \nu \leq 2p^s-1,$ we have 
\vspace{-3mm}\begin{equation*}\mathcal{A}_\rho=\left\{\begin{array}{ll}
1 & \text{ if } \rho=0;\\
0  & \text{ if } 1 \leq \rho \leq n\nu-np^s ;\\
(p^m-1)p^{m(\rho-n\nu+np^s-1)}  & \text{ if }  n\nu-np^s+1 \leq \rho \leq np^s.
\end{array}\right. \vspace{-2mm}\end{equation*}
\vspace{-5mm}\item[(c)] For $ \nu =yp^s$ with $y \in \{0,1 \},$  we have 
\vspace{-3mm}\begin{equation*}\mathcal{A}_\rho=\left\{\begin{array}{ll}
1 & \text{ if } \rho=0;\\
(p^{m(2-y)}-1)p^{m(2-y)(\rho-1)}  & \text{ if } 1 \leq \rho \leq np^s.
\end{array}\right. \vspace{-2mm}\end{equation*}
\vspace{-2mm}\item[(d)] For $1 \leq \nu \leq p^s-1,$  we have 
\vspace{-2mm}\begin{equation*}\mathcal{A}_\rho=\left\{\begin{array}{ll}
1 & \text{ if } \rho=0;\\
(p^m-1)p^{m(\rho-1)}  & \text{ if } 1 \leq \rho \leq n\nu ;\\
(p^{2m}-1)p^{m(2\rho-n\nu-2)} & \text{ if } n\nu+1 \leq \rho \leq np^s.
\end{array}\right. \vspace{-2mm}\end{equation*}
\end{enumerate}
\vspace{-2mm}\end{thm}
\begin{proof} It is easy to see that $\mathcal{A}_{0}=1.$ So from now onwards, throughout the proof, we assume that $1 \leq \rho \leq n p^s.$ \begin{enumerate}
\vspace{-2mm}\item[(a)] When $\nu =2p^s,$ we have $\mathcal{C}=\{ 0\}.$ This gives $\mathcal{A}_\rho=0$ for $1 \leq \rho \leq n p^s.$ 
\vspace{-2mm}\item[(b)] Here by Theorem \ref{dRthm1}, we see that $d_{RT}(\mathcal{C})=n \nu -np^s+1,$ which gives $\mathcal{A}_\rho=0$ for $1 \leq \rho \leq n\nu -np^s.$ Next let $n\nu -np^s+1 \leq \rho \leq np^s.$ Here by Lemma \ref{nilred}(b), we see that $\langle (x^n-\alpha_0)^{p^s} \rangle =\langle \gamma \rangle.$ This implies that $\mathcal{C}=\langle \gamma(x^n-\alpha_0)^{\nu-p^s} \rangle=\{\gamma(x^n-\alpha_0)^{\nu-p^s} F(x) \ | \ F(x) \in \mathcal{T}_{2}[x]\}.$  From this, we observe that the RT weight of the codeword $\gamma(x^n-\alpha_0)^{\nu-p^s} F(x)$ is $\rho$ if and only if  $\text{ deg }F(x)=\rho-n\nu+np^s-1.$ This gives $\mathcal{A}_\rho=(p^m-1)p^{m(\rho-n\nu+np^s-1)}.$
\vspace{-2mm}\item[(c)] Next let $ \nu =y p^s,$ where $y \in \{0,1\}.$ Here by Lemma  \ref{nilred}(b), we see that $\mathcal{C}=\langle (x^n-\lambda_0)^{yp^s} \rangle=\langle \gamma^{y}\rangle= \{\gamma^{y} F(x) \ | \ F(x) \in \mathcal{R}_2[x]\}.$ From this, we see that $\mathcal{A}_\rho =(p^{m(2-y)}-1)p^{m(2-y)(\rho-1)}$ for $1 \leq \rho \leq np^s.$
\vspace{-2mm}\item[(d)]  Here  also, by Lemma \ref{nilred}(b), we note that $\langle (x^n-\alpha_0)^{p^s} \rangle =\langle \gamma \rangle,$ which implies that $\gamma \in \mathcal{C}.$  Further, we observe that any codeword $Q(x)\in \mathcal{C}$ can be uniquely written as $Q(x) =(x^n-\alpha_0)^\nu F_Q(x)+\gamma H_Q(x),$ where $F_{Q}(x), H_{Q}(x) \in \mathcal{T}_{2}[x]$  and $\text{ deg }F_Q(x) \leq n(p^s-\nu)-1$ if $F_Q(x) \neq 0.$  

When $1 \leq \rho \leq n\nu,$ we see that the RT weight of the codeword $Q(x) \in \mathcal{C}$ is $\rho$ if and only if  $F_Q(x)=0$ and $\text{ deg }H_Q(x)=\rho-1.$ From this, we obtain $\mathcal{A}_\rho=(p^m-1)p^{m(\rho-1)}$ for $1 \leq \rho \leq n\nu.$ 

Next let $n\nu+1 \leq \rho \leq np^s.$ Here  we see that the codeword $Q(x)\in \mathcal{C}$ has RT weight $\rho$ if and only if either  (i)  $\text{ deg }F_Q(x)=\rho-n\nu-1$ and $H_Q(x)$ is either 0 or $\text{deg }H_Q(x) \leq \rho-1;$  or  (ii)  $F_Q(x)$ is either 0 or $\text{deg }F_Q(x) \leq \rho-n\nu-2$ and $\text{ deg }H_Q(x) =\rho-1$ holds. From this, we obtain $\mathcal{A}_\rho=(p^m-1)p^{m(2\rho-n\nu-1)}+(p^{m}-1) p^{m(2\rho-n\nu -2)}=(p^{2m}-1)p^{m(2\rho-n\nu-2)}.$ 
\end{enumerate}
\vspace{-8mm}\end{proof}
\vspace{-2mm}From this point on, throughout this section, we assume that $\beta =0.$ 

In the following theorem, we determine all  ideals of the ring $\mathcal{K}_{j}$ when $\beta = 0.$ 
  \vspace{-2mm}\begin{thm}  \label{thm2} When $\beta= 0,$ all the distinct ideals of the ring $\mathcal{K}_{j}$ are as listed below:
\begin{enumerate}\vspace{-2mm}\item[Type I:] (Trivial ideals)
\vspace{-2mm}$$\{0\},   ~~\mathcal{K}_j. \vspace{-2mm}$$
\vspace{-7mm}\item[Type II:] (Principal ideals with non-monic polynomial generators)
\vspace{-2mm} $$ \langle\gamma f_j(x)^\tau  \rangle, \text{ where } 0 \leq \tau  < p^s.\vspace{-2mm}$$
\vspace{-9mm}\item[Type III:] (Principal ideals with monic polynomial generators)
  \vspace{-2mm}$$\langle f_j(x)^{\omega}+\gamma f_j(x)^t G(x)\rangle, \vspace{-4mm}$$
 where $0 < \omega <  p^s,$ $0 \leq t< \kappa$ if $G(x) \neq 0$ and $G(x)$ is either 0 or a unit in $\mathcal{K}_{j}$ of the form $\sum\limits_{i=0}^{\kappa-t-1} a_{i}(x)f_j(x)^{i}$ with $a_{i}(x) \in \mathcal{P}_{d_j}(\mathcal{T}_{2})$ for $0 \leq i \leq \kappa-t-1.$ Here $\kappa$ is the smallest  integer satisfying $0 \leq \kappa \leq \omega$ and $\gamma f_j(x)^{\kappa} \in \langle f_j(x)^{\omega}+\gamma f_j(x)^t G(x)\rangle.$
\vspace{-2mm}\item[Type IV:] (Non-principal ideals)
  \vspace{-2mm}$$\langle f_j(x)^{\omega}+\gamma f_j(x)^t G(x), \gamma f_j(x)^\mu \rangle , \vspace{-1mm}$$ where $0 \leq \mu <\kappa \leq \omega <  p^s,$ $0 \leq t < \mu$ if $G(x) \neq 0$ and $G(x)$ is either 0 or a unit in $\mathcal{K}_{j}$ of the form $\sum\limits_{i=0}^{\mu-t-1} a_{i}(x)f_j(x)^{i},$  $a_{i}(x) \in \mathcal{P}_{d_j}(\mathcal{T}_{2})$ for $0 \leq i \leq \mu-t-1,$ with $\kappa$ as the smallest integer satisfying $ \gamma f_j(x)^{\kappa} \in \langle f_j(x)^{\omega}+\gamma f_j(x)^t G(x)\rangle.$
 \end{enumerate}\vspace{-5mm}
\end{thm}  
\begin{proof} Let $\mathcal{I}$ be a non-trivial ideal of $\mathcal{K}_{j}.$ Now  the following two cases arise: {\bf (i)} $ \mathcal{I} \subseteq\langle\gamma \rangle$ and {\bf (ii)} $ \mathcal{I}\nsubseteq \langle\gamma \rangle.$  
\begin{description}\vspace{-2mm}\item[(i)]   First suppose that $\mathcal{I} \subseteq \langle\gamma \rangle.$ In this case, each element $Q(x) \in \mathcal{I}$ can be uniquely written as $Q(x)=\gamma \sum\limits_{i=0}^{p^s-1}A_{i}^{(Q)}(x){f_j(x)} ^i$, where  $A_{i}^{(Q)}(x) \in \mathcal{P}_{d_j}(\mathcal{T}_{2})$ for $0 \leq i \leq p^s-1.$ Further, for each $Q(x)(\neq 0)\in\mathcal{I},$ we observe that there exists a smallest integer $k_Q$ satisfying  $0\leq k_Q \leq p^{s}-1$ and $A_{k_Q}^{(Q)}(x)\neq 0.$  Let $\tau =\min \{k_Q : Q(x) \in \mathcal{I}\setminus\{0\}\}.$  We note that $0 \leq \tau \leq p^s-1$ and that  there exists $Q_0(x) (\neq 0) \in \mathcal{I}$ such that $k_{Q_0}=\tau,$ i.e., $Q_0(x)=\gamma {f_j(x)} ^{\tau} \sum\limits_{i=\tau}^{p^s-1}A_i^{(Q_0)}(x) {f_j(x)} ^{i-\tau}$ with $A_\tau ^{(Q_0)}(x) \neq 0.$ By Lemma \ref{nilred},  we observe that  $\sum\limits_{i=\tau}^{p^s-1}A_i^{(Q_0)}(x) {f_j(x)} ^{i-\tau}$ is a unit in $\mathcal{K}_{j},$ which implies that $\langle\gamma {f_j(x)} ^\tau\rangle = \langle Q_0(x)\rangle\subseteq\mathcal{I}.$ Moreover, each element $Q(x)\in \mathcal{I}$ can be written as $Q(x)=\gamma {f_j(x)} ^{\tau} \sum\limits_{i=k_Q}^{p^s-1}A_i^{(Q)}(x) {f_j(x)} ^{i-\tau},$ which implies that $\mathcal{I}\subseteq \langle\gamma {f_j(x)} ^\tau \rangle.$ This gives $\mathcal{I}=\langle\gamma {f_j(x)} ^\tau\rangle$ with $0\leq \tau \leq p^{s}-1,$ which is of Type II.
\vspace{-3mm}\item[(ii)] Next suppose that  $\mathcal{I} \not\subseteq \langle\gamma \rangle.$ Here each $Q(x) \in \mathcal{I}$ can be uniquely written as $Q(x)=\sum\limits_{i=0}^{p^s-1}A_i^{(Q)}(x){f_j(x)} ^{i}+\gamma \sum\limits_{\ell=0}^{p^s-1}B_{\ell}^{(Q)}(x) {f_j(x)} ^{\ell},$ where $A_i^{(Q)}(x), B_{\ell}^{(Q)}(x) \in \mathcal{P}_{d_j}(\mathcal{T}_{2})$  for each $i$ and $\ell.$ Now let us define $\mathcal{I}_{1}=\{Q(x) \in \mathcal{I}: A_i^{(Q)}(x) \neq 0 \text{ for some }i, ~0 \leq i \leq p^s-1 \}$ and $\mathcal{I}_{2}=\{Q(x) \in \mathcal{I}: A_i^{(Q)}(x) =0 \text{ for all }i, ~0 \leq i \leq p^s-1\}.$ Since $\mathcal{I} \not\subseteq \langle\gamma \rangle,$ we see that $\mathcal{I}_{1}$ is a non-empty set and $0 \not\in \mathcal{I}_{1}.$ We also observe that  $\gamma\mathcal{I}_{1} \subseteq \mathcal{I}_{2},$  which implies that $\mathcal{I}_{2} \neq \{0\}.$ Further, we note that $\mathcal{I}_{2}$ is a non-zero ideal of $\mathcal{K}_{j}$ and $\mathcal{I}_{2}\subseteq \langle\gamma \rangle.$ This, by case (i), implies that $\mathcal{I}_{2}=\langle\gamma {f_j(x)} ^{\mu}\rangle$ for some integer $\mu,~0 \leq \mu \leq p^s-1.$ Next for each $Q(x)\in \mathcal{I}_{1},$ there exists a smallest integer $\omega_Q$ satisfying $0 \leq \omega_Q \leq p^s-1$ and $A_{\omega_Q}^{(Q)}(x) \neq 0,$ i.e., each $Q(x)\in \mathcal{I}_{1}$ can be written as $Q(x)= f_j(x)^{\omega_Q}W_Q(x)+\gamma M_Q(x),$ where $W_Q(x)=\sum\limits_{i=\omega_Q}^{p^s-1}A_i^{(Q)}(x) {f_j(x)} ^{i-\omega_Q} $ and $M_Q(x)= \sum\limits_{\ell=0}^{p^s-1}B_{\ell}^{(Q)}(x) {f_j(x)} ^{\ell}$  in $\mathcal{K}_{j}.$ By Lemma \ref{nilred}, we see that  $W_Q(x) $ is a unit in $\mathcal{K}_{j}.$ Now let $\omega =\min \{\omega_{Q}: Q(x) \in \mathcal{I}_{1}\}.$ As $\mathcal{I} \neq \mathcal{K}_{j},$ we see that $1 \leq \omega \leq p^s-1.$ Also,  there exists $Q_1(x) \in \mathcal{I}_{1}$ such that $\omega_{Q_1}=\omega,$ i.e., $Q_1(x)= {f_j(x)} ^{\omega} W_{Q_1}(x)+\gamma M_{Q_1}(x),$ where $ W_{Q_1}(x),M_{Q_1}(x) \in \mathcal{K}_{j}.$  For each $Q(x) \in \mathcal{I}_{1},$ we observe that $Q(x)= f_j(x)^{\omega_Q} W_Q(x)+\gamma M_Q(x)=f_j(x)^{\omega_Q-\omega}W_Q(x)Q_1(x)W_{Q_1}(x)^{-1}+\gamma \{M_Q(x)-M_{Q_1}(x)W_{Q_1}(x)^{-1}W_{Q}(x)f_j(x)^{\omega_Q-\omega}\}.$ From this, we see that $\gamma \left(M_Q(x)-M_{Q_1}(x)W_{Q_1}(x)^{-1}W_{Q}(x)f_j(x)^{\omega_Q-\omega}\right) \in \mathcal{I}_2=\langle\gamma f_j(x)^{\mu}\rangle$ for every $Q(x) \in \mathcal{I}_{1}.$ This implies that each  $Q(x) \in \mathcal{I}_{1}$ can be written as $Q(x) =Q_1(x)U_Q(x)+\gamma {f_j(x)} ^{\mu} K_Q(x),$ for some $U_Q(x), K_Q(x) \in \mathcal{K}_{j}.$
From this, we get $\mathcal{I}=\langle Q_1(x), \gamma f_j(x) ^{\mu}\rangle= \langle f_j(x) ^{\omega}W_{Q_1}(x)+\gamma M_{Q_1}(x), \gamma f_j(x) ^{\mu}\rangle.$ As $W_{Q_1}(x)$ is a unit in $\mathcal{K}_{j},$ we obtain $\mathcal{I}=\langle f_j(x) ^{\omega}+\gamma M_{Q_1}(x)W_{Q_1}(x)^{-1}, \gamma f_j(x) ^{\mu}\rangle.$ Let us write $\gamma M_{Q_1}(x)W_{Q_1}(x)^{-1}=\gamma \sum\limits_{i=0}^{p^s-1}\mathcal{G}_i(x) f_j(x) ^i,$ where $\mathcal{G}_i(x) \in \mathcal{P}_{d_j}(\mathcal{T}_{2})$ for $0 \leq i \leq p^s-1.$ For all  $i \geq \mu,$ we note that $\gamma f_j(x) ^i\in \langle\gamma f_j(x)^{\mu}\rangle  (\subseteq \mathcal{I}),$ which implies that  $\mathcal{I}=<f_j(x) ^{\omega}+\gamma \sum\limits_{i=0}^{\mu-1}\mathcal{G}_i(x) f_j(x) ^i, \gamma f_j(x) ^{\mu}>.$ Let us denote $G_1(x)=\sum\limits_{i=0}^{\mu-1}\mathcal{G}_i(x) f_j(x) ^i.$ When $G_1(x) \neq 0,$ there exists a smallest integer $t~(0 \leq t < \mu)$ satisfying $\mathcal{G}_{t}(x) \neq 0$ and we can write   $G_1(x) =f_j(x)^t G(x),$ where $G(x)=\sum\limits_{i=t}^{\mu-1} \mathcal{G}_{i}(x)f_j(x)^{i-t}$ is a unit in $\mathcal{K}_{j}.$ When $G_1(x)=0,$ we choose $G(x)=0.$ From this, we have $\mathcal{I} =\langle f_j(x) ^{\omega}+\gamma f_j(x) ^t G(x),   \gamma f_j(x) ^{\mu}\rangle,$ where $G(x)$ is either 0 or a unit in $\mathcal{K}_{j}$ of the form $\sum\limits_{i=0}^{\mu-t-1} a_{i}(x)f_j(x)^{i}$ with $a_{i}(x) \in \mathcal{P}_{d_j}(\mathcal{T}_{2})$ for $0 \leq i \leq \mu-t-1.$
Since $\kappa$ is the smallest non-negative integer satisfying $\gamma f_j(x)^{\kappa} \in \langle f_j(x)^{\omega}+\gamma f_j(x)^t G(x)\rangle$ and $\gamma f_j(x)^{\omega} \in \langle f_j(x)^{\omega}+\gamma f_j(x)^t G(x)\rangle,$ we get $\kappa \leq \omega.$
As $\gamma f_j(x)^{\kappa} \in \mathcal{I}_{2},$ we must have $\mu \leq  \kappa.$  Moreover, when $\mu=\kappa,$ we note that $\mathcal{I}=\langle f_j(x) ^{\omega}+\gamma f_j(x) ^t G(x)\rangle,$ i.e., $\mathcal{I}$ is of Type III.  In the view of this, we see that for $\mathcal{I} =\langle f_j(x) ^{\omega}+\gamma f_j(x) ^t G(x),   \gamma f_j(x) ^{\mu}\rangle$ to be of Type IV, we must have $\mu < \kappa.$ 
\vspace{-3mm}\end{description} 
This completes the proof of the theorem.
\vspace{-2mm}\end{proof}  
By the above theorem, we see that $\kappa$ is the smallest non-negative integer satisfying $\gamma f_j(x)^{\kappa} \in \langle f_j(x)^{\omega}+\gamma f_j(x)^t G(x)\rangle.$ As $\gamma f_j(x)^{\omega} \in \langle f_j(x) ^{\omega}+\gamma f_j(x) ^t G(x)\rangle,$ we have $\kappa \leq \omega.$ Further,  $\gamma f_j(x)^{\kappa} $ can be written as  
\vspace{-3mm}\begin{equation}\label{rep}\gamma f_j(x)^{\kappa}=\big(f_j(x)^{\omega}+\gamma f_j(x)^t G(x)\big) \big(\sum\limits_{i=0}^{p^s-1}A_i(x)f_j(x)^i+\gamma \sum\limits_{\ell=0}^{p^s-1}B_\ell(x) f_j(x)^\ell\big),\vspace{-3mm}\end{equation} where $A_i(x), B_\ell(x) \in \mathcal{T}_{2}[x]$ for each $i$ and $\ell.$

 In the following proposition, we determine the integer $\kappa$ when $\text{char }\mathcal{R}_2=p$ and $\beta=0.$
\vspace{-2mm}\begin{prop} \label{charp} When $\text{char }\mathcal{R}_2=p$ and $\beta=0,$ we have \vspace{-2mm}\begin{equation*}\kappa=\left\{ \begin{array}{ll} \omega & \text{if } G(x)=0;\\ \min\{\omega, p^s-\omega+t\} & \text{if }G(x)\neq 0 . \end{array}  \right.\vspace{-2mm} \end{equation*}\end{prop}
\begin{proof}  As $f_j(x)^{p^s}=0$ in $\mathcal{K}_{j},$ equation \eqref{rep} can be rewritten as \vspace{-2mm}\begin{equation*}\gamma f_j(x)^{\kappa}=\sum\limits_{i=0}^{p^s-\omega-1}A_i(x)f_j(x)^{i+\omega}+\gamma \sum\limits_{i=0}^{p^s-1}A_i(x)f_j(x)^{i+t}\\ G(x)+\gamma \sum\limits_{\ell=0}^{p^s-1}B_\ell(x) f_j(x)^{\ell+\omega}.
\vspace{-4mm}\end{equation*}This gives $ \sum\limits_{i=0}^{p^s-\omega-1}\overline{A_i(x)}~\overline{f_j(x)}^{i+\omega}=0$ in $\overline{\mathcal{R}_2}[x]/\langle\overline{f_j(x)}^{p^s}\rangle.$ This implies that $ \overline{A_i(x)}=0,$ which further implies that $A_i(x)=0$ for $0 \leq i \leq p^s-\omega-1.$  This gives $\gamma f_j(x)^{\kappa}=\gamma \sum\limits_{i=p^s-\omega}^{p^s-1}A_i(x)f_j(x)^{i+t}G(x)+\gamma \sum\limits_{\ell=0}^{p^s-1}B_\ell(x) f_j(x)^{\ell+\omega}.$ From this, we get $\kappa \geq \omega $ when $G(x)=0$  and $\kappa \geq \min\{\omega,p^s-\omega+t\}$ when $G(x) \neq 0.$ Moreover, when $G(x)\neq 0,$ as $\gamma f_j(x) ^{p^s-\omega+ t}= f_j(x) ^{p^s-\omega} \{f_j(x) ^{\omega}+\gamma f_j(x) ^t G(x)\},$ we get  $p^s-\omega+ t \geq \kappa .$ From this, the desired result follows.\vspace{-2mm}\end{proof}

In the following proposition, we determine the integer $\kappa$ when $\text{char }\mathcal{R}_2=p^2$ and $\beta=0.$
\vspace{-2mm}\begin{prop}\label{charp2} Let $\text{char }\mathcal{R}_2=p^2$ and $\beta=0,$  and  let us write $\gamma (M_j(x)-G(x))=\gamma f_j(x)^{\delta}A_G(x),$ where $0 \leq \delta < p^s-1$ and $A_G(x)$ is either 0 or a unit in $\mathcal{K}_{j}.$ Then we have 
\vspace{-2mm}\begin{equation*}\kappa=\left\{ \begin{array}{ll} \min\{\omega, p^{s-1}\}  & \text{if }G(x)=0;\\ \min\{\omega, p^s-\omega+t,p^{s-1}\} & \text{if } G(x)\neq 0 \text{ and }\omega \neq p^s-p^{s-1}+t; \\
\min\{\omega, p^{s-1}+\delta\}& \text{if }G(x)\neq 0, \omega=p^s-p^{s-1}+t, A_G(x) \neq 0 \text{ and }\delta < p^s-p^{s-1} ;\\
  \omega & \text{if } \omega =p^s-p^{s-1}+t  \text{ with either }  A_G(x)=0  \text{ or }  A_G(x) \neq 0 \text{ and } \delta \geq p^s-p^{s-1}. \end{array}  \right. \end{equation*}
\vspace{-5mm}\end{prop}

\begin{proof}  Since $f_j(x)^{p^s}=-\gamma f_j(x)^{p^{s-1}}M_j(x),$ equation \eqref{rep} can be rewritten as \small \vspace{-2mm}\begin{equation*} \gamma f_j(x)^{\kappa}  =\sum\limits_{i=0}^{p^s-\omega-1}A_i(x)f_j(x)^{i+\omega}-\gamma \sum\limits_{i=p^s-\omega}^{p^s-1}A_i(x)f_j(x)^{i+\omega-p^s+p^{s-1}} M_j(x) +\gamma \sum\limits_{i=0}^{p^s-1}A_i(x)f_j(x)^{i+t}G(x)+ \gamma \sum\limits_{\ell=0}^{p^s-1}B_\ell(x) f_j(x)^{\ell+\omega}.\vspace{-2mm}\end{equation*} \normalsize This gives $ \sum\limits_{i=0}^{p^s-\omega-1}\overline{A_i(x)}~\overline{f_j(x)}^{i+\omega}=0$ in $\overline{\mathcal{R}_2}[x]/\langle\overline{f_j(x)}^{p^s}\rangle.$ This implies that $ \overline{A_i(x)}=0,$ which further implies that $A_i(x)=0$ for $0 \leq i \leq p^s-\omega-1.$  From this, we obtain \vspace{-2mm}\begin{equation}\label{em}\gamma f_j(x)^{\kappa}=-\gamma \sum\limits_{i=p^s-\omega}^{p^s-1}A_i(x)f_j(x)^{i+\omega-p^s+p^{s-1}} M_j(x)+\gamma \sum\limits_{i=p^s-\omega}^{p^s-1}A_i(x)f_j(x)^{i+t}G(x)+\gamma \sum\limits_{\ell=0}^{p^s-1}B_\ell(x) f_j(x)^{\ell+\omega}. \vspace{-2mm}\end{equation} 

When $G(x)=0,$ by \eqref{em}, we get $\kappa \geq \min \{\omega, p^{s-1}\}.$ Further, as $\gamma f_j(x)^{p^{s-1}}M_j(x)=-f_j(x)^{p^s}$ and $\kappa \leq \omega,$  we get   $\kappa = \min \{\omega, p^{s-1}\}.$ 

From now on, throughout the proof, we  assume that $G(x)$ is a unit in $\mathcal{K}_j.$ Here we shall consider the following two cases separately: $p^s-p^{s-1}+t-\omega \neq 0$ and $p^s-p^{s-1}+t-\omega=0.$

First let $p^s-p^{s-1}+t-\omega \neq 0.$ In this case,  we  note that  $\gamma( -f_j(x)^{p^{s-1}} M_j(x)+ f_j(x)^{p^s-\omega+t}G(x))= f_j(x)^{p^s}+\gamma f_j(x)^{p^s-\omega+t}G(x)= f_j(x) ^{p^s-\omega} \{f_j(x) ^{\omega}+\gamma f_j(x) ^t G(x)\} .$ From this and using  the fact that $p^s-p^{s-1}-\omega+t \neq 0,$ we get  $\kappa \leq \min \{p^s-\omega+t, p^{s-1}\}.$  From this and using \eqref{em}, we obtain $\kappa= \min\{\omega,p^s-\omega+t, p^{s-1}\}.$

 Next suppose that $\omega=p^s-p^{s-1}+t .$ In this case, \eqref{em} can be rewritten as
\vspace{-2mm}\begin{equation}\label{em1}\gamma f_j(x)^{\kappa}=-\gamma \sum\limits_{i=p^s-\omega}^{p^s-1}A_i(x)f_j(x)^{i+\omega-p^s+p^{s-1}+\delta}A_G(x) +\gamma \sum\limits_{\ell=0}^{p^s-1}B_\ell(x) f_j(x)^{\ell+\omega}.\vspace{-2mm} \end{equation} 
By \eqref{em1}, we see that $\kappa =\omega$ when $A_G(x)=0.$

Next let $A_G(x)$ be a unit in $\mathcal{K}_j.$ When $\delta \geq p^s-p^{s-1},$  \eqref{em1} becomes $\gamma f_j(x)^{\kappa}=\gamma \sum\limits_{k=0}^{p^s-1}B_k(x) f_j(x)^{k+\omega},$ which gives $\kappa = \omega.$ On the other hand,  we see that $\gamma f_j(x)^{p^{s-1}+\delta} A_G(x)=\gamma f_j(x)^{p^{s-1}} (M_j(x)-G(x)) ,$ which gives   $\kappa \leq p^{s-1}+\delta$ 
when $ \delta< p^s-p^{s-1}.$   From this and using \eqref{em1}, we get $\kappa = \min\{\omega, p^{s-1}+\delta \}$ when $\delta < p^s-p^{s-1}.$ \vspace{-1mm}\end{proof}

In the following theorem, we determine cardinalities of all ideals of $\mathcal{K}_{j}$ when $\beta=0.$
\vspace{-2mm}\begin{thm}\label{thm3} Suppose that $\beta=0.$ Let $\mathcal{I}$ be an ideal of $\mathcal{K}_j$ (as determined in Theorem \ref{thm2}). \begin{enumerate} \vspace{-2mm}\item[(a)] If $\mathcal{I}=\{0\},$ then    $|\mathcal{I}|=1.$ 
\vspace{-2mm}\item[(b)] If $\mathcal{I}=\mathcal{K}_{j},$ then  $|\mathcal{I}|=p^{2md_jp^s}.$
\vspace{-2mm}\item[(c)] If $\mathcal{I}=\langle\gamma f_j(x) ^{\tau}\rangle$ is of Type II, then $|\mathcal{I}|=p^{md_j(p^s-\tau)}.$
\vspace{-2mm}\item[(d)] Let $\mathcal{I}=\langle f_j(x) ^{\omega}+\gamma f_j(x) ^{t}G(x)\rangle$ be of Type III. Let us write $\gamma (M_j(x)-G(x)) =\gamma f_j(x)^{\delta} A_G(x),$ where $0 \leq \delta < p^s$ and $A_G(x)$ is either 0 or a unit in $\mathcal{K}_{j}.$ When $\text{char }\mathcal{R}_2=p,$  we have
\vspace{-2mm}\begin{equation*}|\mathcal{I}|=\left\{\begin{array}{ll} p^{2md_j(p^s-\omega)} & \text{if } \text{either } G(x)=0 \text{ or }   G(x)\neq0 \text{ and }p^s-2\omega+t \geq 0;\\
p^{md_j(p^s-t)} & \text{if }  G(x)\neq0 \text{ and } p^s-2\omega+t < 0.
\end{array}\right.\vspace{-2mm}\end{equation*}
When $\text{char }\mathcal{R}_2=p^2,$ we have
\vspace{-2mm}\begin{equation*}|\mathcal{I}|=\left\{\begin{array}{ll} p^{2md_j(p^s-\omega)} & \text{if } \text{either }G(x)=0, \omega \leq p^{s-1} \text{ or } G(x)\neq0, p^s-2\omega+t \geq 0, \omega \leq p^{s-1}, \\& \omega \neq  p^s-p^{s-1} +t   \text{ or }   A_G(x)=0,  \omega =  p^s-p^{s-1} +t \text{ or }   A_G(x)\neq 0, \\&  \omega =p^s-p^{s-1} +t,  \delta \geq p^s -p^{s-1} \text{ or }   A_G(x)\neq0, \omega =   p^s-p^{s-1} +t \leq p^{s-1}+\delta  ;\\
p^{md_j(p^s-t)} & \text{if }  G(x)\neq 0,  p^s-2\omega+t \leq 0,  p^s- p^{s-1}-\omega +t <0;\\
 p^{md_j(2p^s-\omega-p^{s-1})} & \text{if } G(x)=0, \omega > p^{s-1}  \text{ or }  G(x)\neq0,  \omega \geq p^{s-1}, p^s- p^{s-1}-\omega +t >0;\\p^{md_j(2p^s-\omega-p^{s-1}-\delta)}&  \text{if }   A_G(x)\neq0,  \omega=p^s-p^{s-1} +t, \omega> p^{s-1}+\delta.
\end{array}\right.\vspace{-2mm}\end{equation*}

\vspace{-2mm}\item[(e)] If $\mathcal{I}=\langle f_j(x) ^{\omega}+\gamma f_j(x) ^{t}G(x), \gamma f_j(x) ^{\mu}\rangle$ is of Type IV, then $|\mathcal{I}|= p^{md_j(2p^s-\mu-\omega)}.$
\end{enumerate}\end{thm}
\vspace{-4mm}\begin{proof} To prove this, we see, by Lemma \ref{card}, that $|\mathcal{I}|=|Res_{\gamma}(\mathcal{I})||Tor_{\gamma}(\mathcal{I})|.$ So we need to determine cardinalities of  $Res_{\gamma}(\mathcal{I})$ and $ Tor_{\gamma}(\mathcal{I}),$ which are ideals of the quotient ring $\overline{\mathcal{R}_2}[x]/\langle\overline{ f_j(x)}^{p^s}\rangle.$  To do this, we first note that the nilpotency index of $\overline{f_j(x)}$ in $\overline{\mathcal{R}_2}[x]/\langle \overline{ f_j(x)}^{p^s}\rangle$ is $p^s.$ Further,  by Proposition \ref{pr1}, we observe that $\overline{\mathcal{R}_2}[x]/\langle \overline{ f_j(x)}^{p^s}\rangle$ is a finite commutative chain ring with unity and all its ideals are given by $\langle \overline{f_j(x)}^{i}\rangle$ for $0 \leq i \leq p^s.$ We also observe that  the residue field of  $\overline{\mathcal{R}_2}[x]/\langle \overline{ f_j(x)}^{p^s}\rangle$ is of order $p^{md_j}.$  This, by Proposition \ref{pr1} again, implies that  \vspace{-1mm}\begin{equation}\label{eee}| \langle \overline{f_j(x)}^{i}\rangle|=p^{md_j(p^s-i)}\text{  for }0 \leq i \leq p^s.\vspace{-2mm}\end{equation}
\begin{enumerate}\vspace{-3mm} \item[(a)] If $\mathcal{I}=\{0\},$ then   $\text{Res}_{\gamma}(\mathcal{I})=Tor_{\gamma} ({\mathcal{I}})=\{0\},$ which gives $|\mathcal{I}|=1.$ 
\vspace{-2mm}\item[(b)] If $\mathcal{I}=\mathcal{K}_{j},$ then   $\text{Res}_{\gamma}(\mathcal{I})=Tor_{\gamma}({\mathcal{I}})=\langle1\rangle=\overline{\mathcal{R}_2}[x]/\langle\overline{ f_j(x)}^{p^s}\rangle.$ From this and using \eqref{eee}, we get $|\mathcal{I}|=p^{2md_jp^s}.$  
\vspace{-6mm}\item[(c)] If $\mathcal{I}=\langle\gamma f_j(x)^{\tau }\rangle$ is of Type II, then $\text{Res}_{\gamma}(\mathcal{I})=\{0\}$  and   $Tor_{\gamma}({\mathcal{I}})=\langle\overline{ f_j(x)}^\tau \rangle.$ From this and using \eqref{eee}, we obtain $|\mathcal{I}|=p^{md_j(p^s-\tau)}.$
\vspace{-2mm}\item[(d)] If $\mathcal{I}=\langle f_j(x)^{\omega}+\gamma f_j(x)^{t}G(x)\rangle$ is of Type III, then it is easy to see that  $\text{Res}_{\gamma}(\mathcal{I})=\langle \overline{f_j(x)}^{\omega}\rangle$ and  $Tor_{\gamma}({\mathcal{I}})=\langle\overline{ f_j(x)}^{\kappa}\rangle .$ Now by applying Propositions \ref{charp} and \ref{charp2} and using \eqref{eee}, part (d) follows.
\vspace{-2mm}\item[(e)] If $\mathcal{I}=\langle f_j(x)^{\omega}+\gamma f_j(x)^{t}G(x), \gamma f_j(x)^{\mu} \rangle$ is of Type IV, then   $\text{Res}_{\gamma}(\mathcal{I})=\langle\overline{f_j(x)}^{\omega}\rangle$ and $Tor_{\gamma}({\mathcal{I}})= \langle\overline{ f_j(x)}^{\mu}\rangle.$  From this and using \eqref{eee}, we get $|\mathcal{I}|= p^{md_j(2p^s-\mu-\omega)}.$ 
\end{enumerate}
\vspace{-4mm}\end{proof}
\vspace{-5mm}In the following theorem,  we  determine the orthogonal complement of each ideal of  $\mathcal{K}_j$ when $\beta=0.$ 
\vspace{-2mm}\begin{thm}\label{dual} Suppose that $\beta=0.$ Let $\mathcal{I}$ be an ideal of $\mathcal{K}_j$ (as determined in Theorem \ref{thm2}).
\begin{enumerate}\vspace{-2mm} \item[(a)] If $\mathcal{I}=\{0\},$ then  $\mathcal{I}^{\perp}=\widehat{\mathcal{K}_{j}}=\mathcal{R}_2[x]/\langle k_j^*(x)\rangle.$ 
\vspace{-2mm}\item[(b)] If $\mathcal{I}=\mathcal{K}_{j},$ then   $\mathcal{I}^{\perp}=\{0\}.$ 
\vspace{-2mm}\item[(c)] If $\mathcal{I}=\langle\gamma f_j(x) ^{\tau}\rangle$ is of Type II, then $\mathcal{I}^{\perp}=\langle f_j^*(x) ^{p^s-\tau},\gamma \rangle.$ 
\vspace{-2mm}\item[(d)] Let $\mathcal{I}=\langle f_j(x) ^{\omega}+\gamma f_j(x) ^{t}G(x)\rangle$ be of Type III. Let us write $\gamma (M_j(x)-G(x)) =\gamma f_j(x)^{\delta} A_G(x),$ where $0 \leq \delta < p^s$ and $A_G(x)$ is either 0 or a unit in $\mathcal{K}_{j}.$ 
\vspace{-2mm}\begin{enumerate}\item[(i)] When $\text{char }\mathcal{R}_2=p,$  we have 
\vspace{-2mm}\begin{equation*} \mathcal{I}^{\perp}=\left\{ \begin{array}{ll} 
\langle f_j^*(x) ^{p^s-\omega}\rangle  \text{ if } G(x)=0;\vspace{1mm}\\
\langle f_j^*(x) ^{p^s-\omega}-\gamma x^{d_j\omega-d_jt-\text{deg }G(x)}f_j^*(x) ^{p^s-2\omega+t}G^*(x)\rangle  \text{ if } G(x) \neq 0 \text{ and } p^s-2\omega+t \geq 0 ;\vspace{1mm}\\
\langle f_j^*(x) ^{\omega-t}-\gamma x^{d_j\omega-d_jt-\text{deg }G(x)}G^*(x), \gamma f_j^*(x) ^{p^s-\omega} \rangle  \text{ if } G(x) \neq 0 \text{ and } p^s-2\omega+t < 0.\vspace{1mm}\end{array}\right.
\vspace{-2mm}\end{equation*}
\vspace{-4mm}\item[(ii)] Let $\text{char }\mathcal{R}_2=p^2.$ When $G(x) \neq 0,$ $t =\omega-p^s+p^{s-1}$ with either $A_G(x)=0$ or $A_G(x)\neq 0$ and  $\delta\geq  p^s-p^{s-1},$ we have $\mathcal{I}=\langle f_j(x) ^{\omega}+\gamma f_j(x) ^{\omega-p^s+p^{s-1}}M_j(x)\rangle.$ Furthermore, we have 
\vspace{-2mm}\begin{equation*} \mathcal{I}^{\perp}=\left\{ \begin{array}{ll} 
\langle f_j^*(x)^{p^s-\omega}+ \gamma x^{d_jp^s-d_jp^{s-1}-\text{deg }M_j(x)} f_j^*(x)^{p^{s-1}-\omega}M_j^*(x)\rangle  \text{ if }G(x)=0 \text{ and } \omega \leq p^{s-1};\vspace{1mm}\\
\langle f_j^*(x)^{p^s-p^{s-1}}+ \gamma x^{d_jp^s-d_jp^{s-1}-\text{deg }M_j(x)} M_j^*(x), \gamma f_j^*(x)^{p^s-\omega}\rangle  \text{ if } G(x)=0  \text{ and } \omega > p^{s-1};\vspace{1mm}\\
\langle f_j^*(x)^{p^s-p^{s-1}} +\gamma x^{d_jp^s-d_jp^{s-1}-\text{deg }M_j(x)} M_j^*(x)-\gamma x^{d_j\omega-d_jt-\text{deg }G(x)} f_j^*(x)^{p^s+t-\omega-p^{s-1}}G^*(x), \\ \gamma f_j^*(x)^{p^s-\omega} \rangle \text{ if } G(x) \neq 0,  p^s-p^{s-1}+t-\omega> 0 \text{ and } \omega > p^{s-1} ;\vspace{1mm}\\
\langle f_j^*(x)^{p^s-\omega} +\gamma x^{d_jp^s-d_jp^{s-1}-\text{deg }M_j(x)} f_j^*(x)^{p^{s-1}-\omega} M_j^*(x)-\gamma x^{d_j\omega-d_jt-\text{deg }G(x)} f_j^*(x)^{p^s+t-2\omega}G^*(x)\rangle \\ \text{ if } G(x) \neq 0, p^s-p^{s-1}+t-\omega> 0 \text{ and } \omega \leq p^{s-1};\vspace{1mm}\\
\langle f_j^*(x)^{p^s-\omega} +\gamma x^{d_jp^s-d_jp^{s-1}-\text{deg }M_j(x)} f_j^*(x)^{p^{s-1}-\omega} M_j^*(x)-\gamma x^{d_j\omega-d_jt-\text{deg }G(x)} f_j^*(x)^{p^s+t-2\omega}G^*(x)\rangle \\ \text{ if }G(x) \neq 0, p^s-p^{s-1}+t-\omega< 0 \text{ and } p^s-2\omega+t \geq 0;\vspace{1mm}\\
\langle f_j^*(x)^{\omega-t} +\gamma x^{d_jp^s-d_jp^{s-1}-\text{deg }M_j(x)} f_j^*(x)^{p^{s-1} +\omega-t-p^s} M_j^*(x)-\gamma x^{d_j\omega-d_jt-\text{deg }G(x)} G^*(x), \\\gamma f_j^*(x)^{p^s-\omega} \rangle \text{ if }G(x) \neq 0, p^s-p^{s-1}+t-\omega< 0 \text{ and } p^s-2\omega+t < 0;\vspace{1mm}\\
\langle f_j^*(x)^{p^s-\omega}\rangle  \text{ if } G(x) \neq 0, \omega=p^s-p^{s-1}+t  \text{ with either }A_G(x)=0  \text{ or }  A_G(x) \neq 0, \delta\geq  p^s-p^{s-1};\vspace{1mm}\\
\langle f_j^*(x)^{p^s-\omega}+\gamma x^{d_jp^s-d_jp^{s-1}-d_j\delta-\text{deg }A_G(x)}f_j^*(x)^{p^{s-1}-\omega+\delta}A_G^*(x)\rangle  \text{ if }G(x) \neq 0, t=\omega-p^s+p^{s-1},\\ A_G(x) \neq 0  \text{ and }   \omega - p^{s-1} \leq \delta < p^s -p^{s-1};\\
\langle f_j^*(x)^{p^s-p^{s-1}-\delta}+\gamma x^{d_jp^s-d_jp^{s-1}-d_j\delta-\text{deg }A_G(x)}A_G^*(x), \gamma f_j^*(x)^{p^s-\omega}\rangle  \text{ if } G(x) \neq 0, t=\omega-p^s+p^{s-1},\\ A_G(x)\neq 0  \text{ and }  \delta<  \omega - p^{s-1}. \end{array}\right.
\end{equation*}\end{enumerate}
\vspace{-4mm}\item[(e)] Let $\mathcal{I}=\langle f_j(x) ^{\omega}+\gamma f_j(x)^tG(x), \gamma f_j(x) ^{\mu}\rangle$ be of Type IV. Let us write $\gamma (M_j(x)-G(x)) =\gamma f_j(x)^{\delta} A_G(x),$ where $0 \leq \delta < p^s$ and $A_G(x)$ is either 0 or a unit in $\mathcal{K}_{j}.$ 
\vspace{-2mm}\begin{enumerate}\item[(i)]  When $\text{char }\mathcal{R}_2=p,$  we have\vspace{-2mm} \begin{equation*} \mathcal{I}^{\perp}=\left\{\begin{array}{ll}
\langle f_j^*(x)^{p^s-\mu}, \gamma f_j^*(x)^{p^s-\omega}\rangle  \text{ if }  G(x)=0;\vspace{1mm}\\
\langle f_j^*(x)^{p^s-\mu}-\gamma x^{d_j\omega-d_jt-\text{deg }G(x)}f_j^*(x)^{p^s-\mu-\omega+t}G^*(x), \gamma f_j^*(x)^{p^s-\omega}\rangle  \text{ if } G(x) \neq 0;\vspace{1mm}\\
\end{array}\right. \vspace{-3mm}\end{equation*}
\vspace{-4mm}\item[(ii)] Let $\text{char }\mathcal{R}_2=p^2.$ When $G(x) \neq 0,$ $t =\omega-p^s+p^{s-1}$ with either $A_G(x)=0$ or $A_G(x)\neq 0$ and  $\delta\geq  p^s-p^{s-1},$ we have $\mathcal{I}=\langle f_j(x) ^{\omega}+\gamma f_j(x) ^{\omega-p^s+p^{s-1}}M_j(x), \gamma f_j(x) ^{\mu}\rangle.$ Furthermore,  we have 
\vspace{-2mm} \begin{equation*} \mathcal{I}^{\perp}=\left\{\begin{array}{ll}
\langle f_j^*(x)^{p^s-\mu}, \gamma f_j^*(x)^{p^s-\omega}\rangle  \text{ if }  G(x)=0\text{ and } p^s-p^{s-1}-\omega+\mu\leq 0;\vspace{1mm}\\
\langle f_j^*(x)^{p^s-\mu}+ \gamma x^{d_jp^s-d_jp^{s-1}-\text{deg }M_j(x)} f_j^*(x)^{p^{s-1}-\mu}M_j^*(x), \gamma f_j^*(x)^{p^s-\omega}\rangle  \text{ if } G(x)=0 \text{ and }\\ p^s-p^{s-1}-\omega+\mu> 0;\vspace{1mm}\\
\langle f_j^*(x)^{p^s-\mu}+\gamma x^{d_jp^s-d_jp^{s-1}-\text{deg }M_j(x)}f_j^*(x)^{p^{s-1}-\mu}M_j^*(x)-\gamma x^{d_j\omega-d_jt-\text{deg }G(x)}f_j^*(x)^{p^{s}-\mu+t-\omega}G^*(x),\\  \gamma f_j^*(x)^{p^s-\omega} \rangle  \text{ if } G(x) \neq 0 \text{ and } p^s-p^{s-1}+t-\omega \neq 0;\vspace{1mm}\\
\langle f_j^*(x)^{p^s-\mu} ,\gamma f_j^*(x)^{p^s-\omega}\rangle  \text{ if } G(x) \neq 0, t=\omega-p^s+p^{s-1} \text{ with either }A_G(x) =0 \text{ or } A_G(x) \neq 0,\\  \delta \geq p^s-p^{s-1}; \vspace{1mm}\\
\langle f_j^*(x)^{p^s-\mu}+\gamma x^{d_jp^s-d_jp^{s-1}-d_j\delta-\text{deg }A_G(x)}f_j^*(x)^{p^{s-1}-\mu+\delta}A_G^*(x), \gamma f_j^*(x)^{p^s-\omega}\rangle \text{ if } G(x) \neq 0,\\ t=\omega-p^s+p^{s-1}, A_G(x) \neq 0 \text{ and }  \delta < p^s-p^{s-1} .
\end{array}\right. \vspace{-2mm}\end{equation*}
\end{enumerate}\end{enumerate}\end{thm}
\begin{proof} It is easy to see that $\mathcal{I}^{\perp}=\mathcal{K}_{j}$ when $\mathcal{I}=\{0\}$ and that $\mathcal{I}^{\perp}=\{0\}$ when $\mathcal{I}=\mathcal{K}_{j}.$  As the proofs of parts (d) and (e) are almost similar, we will prove parts (c) and (e) only. For this, we see,  by Lemma \ref{nilred}(c),  that the nilpotency index $\mathfrak{N}$ of $ f_j(x)$ in $\mathcal{K}_j$ is given by  $\mathfrak{N}=p^s$ when $\text{char }\mathcal{R}_2=p$ and  is given by  $\mathfrak{N}=2p^s-p^{s-1}$ when $\text{char }\mathcal{R}_2=p^2.$ We also note that $\gamma f_j(x)^{p^s}=0$ and $\gamma f_j(x)^{p^s-1}\neq0$ in $\mathcal{K}_j.$

To prove (c),  suppose that $\mathcal{I}=\langle\gamma f_j(x) ^{\tau}\rangle$ is of Type II. Here we see that $\text{ann}(\mathcal{I})=\langle f_j(x) ^{p^s-\tau},\gamma \rangle.$ From this and  by applying Lemma \ref{lemdual},  we obtain $\mathcal{I}^{\perp}=\langle f_j^*(x) ^{p^s-\tau},\gamma \rangle.$
 
To prove (e),  let $\mathcal{I}=\langle f_j(x) ^{\omega}+\gamma f_j(x) ^{t}G(x), \gamma f_j(x) ^{\mu}\rangle$ be of Type IV. Here we shall distinguish the following two cases: (i) $\text{char }\mathcal{R}_2=p$ and (ii) $\text{char }\mathcal{R}_2=p^2.$

\vspace{-1mm}\begin{enumerate}\vspace{-2mm}\item[(i)] First let $\text{char }\mathcal{R}_2=p.$ 

\vspace{-1mm}When $G(x)=0,$ we have $\mathcal{I}=\langle f_j(x) ^{\omega}, \gamma f_j(x) ^{\mu} \rangle.$ It is easy to see that $\text{ann}(\mathcal{I})=  \langle f_j(x)^{p^s-\mu}, \gamma f_j(x) ^{p^s-\omega}\rangle ,$ which, by applying Lemma \ref{lemdual},  implies that $\mathcal{I}^{\perp}=\langle f_j^*(x) ^{p^s-\mu}, \gamma f_j^*(x) ^{p^s-\omega}\rangle.$

 Now suppose that $G(x)$ is a unit in $\mathcal{K}_j.$ As $\text{ann}(\mathcal{I})$ is an ideal of $\mathcal{K}_j,$ by Theorem \ref{thm2}, we can write $\text{ann}(\mathcal{I})=\langle f_j(x) ^{a}+\gamma f_j(x) ^{b}H(x), \gamma f_j(x) ^{c}\rangle,$ where $H(x)$ is either $0$ or a unit in $\mathcal{K}_j.$ This implies that \vspace{-2mm}\begin{equation}\label{eq1}\gamma f_j(x) ^{a+\mu}=0,\gamma f_j(x) ^{c+\omega}=0 \text{ and } f_j(x) ^{a+\omega}+\gamma( f_j(x) ^{a+t}G(x)+ f_j(x) ^{b+\omega}H(x))=0.\vspace{-2mm}\end{equation}
It is easy to observe that  \eqref{eq1} holds for all $a \geq \max\{p^s-\mu,\omega -t\},$  $c \geq p^s-\omega,$ $b=a+t-\omega $ and  $H(x)=-G(x).$  Further, by Theorem \ref{thm2} and Proposition \ref{charp}, we see that $\mu < \kappa$  and $\kappa=\min\{\omega, p^s-\omega+t\},$ which implies that $p^s-\mu \geq \omega -t.$ This implies that  $\text{ann}(\mathcal{I})=\langle f_j(x) ^{p^s-\mu}-\gamma f_j(x) ^{p^s-\mu-\omega+t}G(x), \gamma f_j(x) ^{p^s-\omega}\rangle.$ From this and by applying Lemma \ref{lemdual}, we obtain
$\mathcal{I}^{\perp} = \langle f_j^*(x)^{p^s-\mu}-\gamma x^{d_j\omega-d_jt-\text{deg }G(x)} f_j^*(x)^{p^s-\mu-\omega+t}G^*(x), \gamma f_j^*(x)^{p^s-\omega}\rangle.$
\vspace{-6mm}\item[(ii)] Next suppose that $\text{char } \mathcal{R}_2=p^2.$ Here we have $f_j(x)^{p^s}=-\gamma f_j(x)^{p^{s-1}}M_j(x)$ in $\mathcal{K}_{j}.$  

When $G(x)=0,$ we have $\mathcal{I}=\langle f_j(x) ^{\omega}, \gamma f_j(x) ^{\mu} \rangle.$ As $\text{ann}(\mathcal{I})$ is an ideal of $\mathcal{K}_j,$  by Theorem \ref{thm2}, we can write $\text{ann}(\mathcal{I})=\langle f_j(x) ^{a}+\gamma f_j(x) ^{b}H(x), \gamma f_j(x) ^{c}\rangle,$ where $H(x)$ is either $0$ or a unit in $\mathcal{K}_j.$ This implies that \vspace{-2mm}\begin{equation}\label{eq2}\gamma f_j(x) ^{a+\mu}=0,\gamma f_j(x) ^{c+\omega}=0 \text{ and } f_j(x) ^{a+\omega}+\gamma f_j(x) ^{b+\omega}H(x)=0. 
\vspace{-2mm}\end{equation}
By Theorem \ref{thm2} and Proposition \ref{charp2}, we see that $\mu < \kappa$ and $\kappa =\min\{\omega, p^{s-1}\},$ which implies that $\mu < \omega$ and $\mu < p^{s-1}.$ Using this and by \eqref{eq2}, we get $a \geq p^s-\mu,$  $c \geq p^s-\omega$ and $-\gamma f_j(x) ^{a+\omega-p^s+p^{s-1}}M_j(x)+\gamma f_j(x) ^{b+\omega}H(x)=0,$ which holds only if $a \geq \max\{p^s-\mu,p^s-p^{s-1}\}=p^s-\mu,$ $b=a-p^s+p^{s-1}$ and $\gamma H(x)=\gamma M_j(x).$ This implies that  $\text{ann}(\mathcal{I})=\langle f_j(x) ^{p^s-\mu}, \gamma f_j(x) ^{p^s-\omega}\rangle$ when $p^s-p^{s-1}-\omega+\mu\leq0,$ while $\text{ann}(\mathcal{I})=\langle f_j(x) ^{p^s-\mu}+\gamma f_j(x) ^{p^{s-1}-\mu}M_j(x), \gamma f_j(x) ^{p^s-\omega}\rangle$ when $p^s-p^{s-1}-\omega+\mu>0.$ From this and by applying Lemma \ref{lemdual}, we get $\mathcal{I}^{\perp} = \langle f_j^*(x)^{p^s-\mu}, \gamma f_j^*(x)^{p^s-\omega}\rangle$ when $p^s-p^{s-1}-\omega+\mu\leq0,$ while $\mathcal{I}^{\perp} = \langle f_j^*(x)^{p^s-\mu}+\gamma x^{d_jp^s-d_jp^{s-1}-\text{deg }M_j(x)}f_j^*(x)^{p^{s-1}-\mu}M_j^*(x), \gamma f_j^*(x)^{p^s-\omega}\rangle$ when $p^s-p^{s-1}-\omega+\mu>0.$
 
 Next assume that $G(x)$ is a unit in $\mathcal{K}_j.$ As $\text{ann}(\mathcal{I})$ is an ideal of $\mathcal{K}_{j},$ by Theorem \ref{thm2}, we can write $\text{ann}(\mathcal{I})=\langle f_j(x) ^{a}+\gamma f_j(x) ^{b}H(x), \gamma f_j(x) ^{c}\rangle,$ where $H(x)$ is either $0$ or a unit in $\mathcal{K}_j.$ This implies that \vspace{-1mm}\begin{equation}\label{eq3}\gamma f_j(x) ^{a+\mu}=0,\gamma f_j(x) ^{c+\omega}=0 \text{ and } f_j(x) ^{a+\omega}+\gamma(f_j(x)^{a+t}G(x)+ f_j(x) ^{b+\omega}H(x))=0. 
\vspace{-1mm}\end{equation}
From \eqref{eq3}, we get   $a \geq p^s-\mu,$ $c \geq p^s-\omega$ and  \vspace{-1mm}\begin{equation}\label{eq4}  \gamma (-f_j(x) ^{a+\omega-p^s+p^{s-1}}M_j(x)+f_j(x)^{a+t}G(x)+ f_j(x) ^{b+\omega}H(x))=0.\end{equation} \vspace{-1mm}
Here we consider the following two cases separately: $p^s-p^{s-1}+t-\omega \neq 0$ and $p^s-p^{s-1}+t-\omega=0.$

When $p^s-p^{s-1}+t-\omega \neq 0,$ by Proposition \ref{charp2}, we note that $p^s-\omega+t-\mu \geq 0$ and $\mu \leq p^{s-1}.$ In this case, we see that \eqref{eq4} holds for $a=p^s-\mu,$ which implies that $\text{ann}(\mathcal{I})=\langle f_j(x) ^{p^s-\mu}+\gamma f_j(x) ^{p^{s-1}-\mu}M_j(x)-\gamma f_j(x)^{p^s-\mu+t-\omega} G(x), \gamma f_j(x) ^{p^s-\omega}\rangle.$ From this and using Lemma \ref{lemdual}, we get
\small\begin{equation*}\mathcal{I}^{\perp} = \langle f_j^*(x)^{p^s-\mu}+\gamma x^{d_jp^s-d_jp^{s-1}-\text{deg }M_j(x)}f_j^*(x)^{p^{s-1}-\mu}M_j^*(x)-\gamma x^{d_j\omega-d_jt-\text{deg }G(x)}f_j^*(x)^{p^{s}-\mu+t-\omega}G^*(x),   \gamma f_j^*(x)^{p^s-\omega}\rangle.\vspace{-2mm}\end{equation*}
 
Next suppose that $p^s-p^{s-1}+t-\omega=0.$  In this case, \eqref{eq4} can be rewritten as\vspace{-1mm}  \begin{equation}\label{eq5}\gamma f_j(x) ^{b+\omega}H(x)=\gamma f_j(x) ^{a+\omega-p^s+p^{s-1}}(M_j(x)-G(x))=\gamma f_j(x) ^{a+\omega-p^s+p^{s-1}+\delta}A_G(x).\vspace{-2mm}\end{equation}   
When $A_G(x)=0,$ we see that \eqref{eq5} holds for all $b \geq p^s-\omega,$ which implies that   $\text{ann}(\mathcal{I})=\langle f_j(x) ^{p^s-\mu}, \gamma f_j(x) ^{p^s-\omega}\rangle.$ From this and using Lemma \ref{lemdual}, we obtain
$\mathcal{I}^{\perp} = \langle f_j^*(x)^{p^s-\mu}, \gamma f_j^*(x)^{p^s-\omega}\rangle.$ 

Next let  $A_G(x)$ be a unit in $\mathcal{K}_j.$  Here by Proposition \ref{charp2}, we see that  $p^{s-1}+\delta \geq \mu.$  When $\delta\geq  p^s-p^{s-1},$ we see that $\gamma f_j(x) ^{\omega+\delta-p^s+p^{s-1}} \in \mathcal{I},$ which implies that $\mathcal{I}=\langle f_j(x) ^{\omega}+\gamma f_j(x) ^{\omega-p^s+p^{s-1}}M_j(x)-\gamma f_j(x) ^{\omega+\delta-p^s+p^{s-1}}A_G(x), \gamma f_j(x) ^{\mu}\rangle=\langle f_j(x) ^{\omega}+\gamma f_j(x) ^{\omega-p^s+p^{s-1}}M_j(x), \gamma f_j(x) ^{\mu}\rangle.$ Now when $\delta \geq p^s-p^{s-1},$ we see that \eqref{eq5} holds for all $b \geq p^s-\omega,$ which implies that   $\text{ann}(\mathcal{I})=\langle f_j(x) ^{p^s-\mu}, \gamma f_j(x) ^{p^s-\omega}\rangle.$ From this and using Lemma \ref{lemdual}, we get $\mathcal{I}^{\perp} = \langle f_j^*(x)^{p^s-\mu}, \gamma f_j^*(x)^{p^s-\omega}\rangle.$ Moreover, when $\delta < p^s-p^{s-1},$ we observe that  \eqref{eq5} holds for all $a \geq p^s-\mu,$ $b=a+p^{s-1}-p^s+\delta$ and $H(x)=A_G(x).$  This implies that $\text{ann}(\mathcal{I})=\langle f_j(x) ^{p^s-\mu}+\gamma f_j(x) ^{p^{s-1}-\mu+\delta}A_G(x), \gamma f_j(x) ^{p^s-\omega}\rangle,$ which, by Lemma \ref{lemdual}, further implies that 
$\mathcal{I}^{\perp} = \langle f_j^*(x)^{p^s-\mu}+\gamma x^{d_jp^s-d_jp^{s-1}-d_j\delta-\text{deg }A_G(x)}f_j^*(x)^{p^{s-1}-\mu+\delta}A_G^*(x), \gamma f_j^*(x)^{p^s-\omega}\rangle.$ 
 \end{enumerate}\vspace{-2mm}
This completes the proof of the theorem.
\vspace{-1mm}\end{proof}
\vspace{-1mm}As a consequence of the above theorems, we obtain some isodual $\alpha$-constacyclic codes of length $np^s$ over $\mathcal{R}_2.$
\vspace{-2mm}\begin{cor} \label{c2} Suppose that $\beta=0.$ Let $\alpha =\alpha_0^{p^s}\in \mathcal{T}_{2} \setminus \{0\},$ where $\alpha_0 \in \mathcal{T}_{2}$ is such that $x^n-\alpha_0$ is  irreducible  over $\mathcal{R}_2 .$  Following the same notations as in Theorem \ref{thm2}, we have the following:\vspace{-1mm}
\begin{enumerate}
\vspace{-1mm}\item[(a)] The code $\langle \gamma \rangle$ is the only isodual $\alpha$-constacyclic code of Type II over $\mathcal{R}_2.$ 
\vspace{-2mm}\item[(b)] When $\text{ char }\mathcal{R}_2$ is an odd prime $p$  with either $G(x)=0$ or $G(x)\neq 0$ and $p^s-2\omega+t \geq 0,$ there does not exist any isodual $\alpha$-constacyclic code of Type III over $\mathcal{R}_2.$   

When $\text{ char }\mathcal{R}_2=2,$  all the isodual $\alpha$-constacyclic codes of  Type III over $\mathcal{R}_2$ are given by $ \langle(x^n-\alpha_0) ^{2^{s-1}} \rangle, $ $\langle(x^n-\alpha_0) ^{2^{s-1}}+\gamma (x^n-\alpha_0) ^{t}G(x)\rangle$ and $\langle(x^n-\alpha_0) ^{\omega}+\gamma G(x)\rangle,$ where  $ \omega > 2^{s-1}$ and $G(x) \neq 0.$ 

When $\text{char }\mathcal{R}_2=4,$  the  codes $\langle(x^n-\alpha_0) ^{2^{s-1}}\rangle$ and $\langle(x^n-\alpha_0) ^{2^{s-1}}+\gamma (x^n-\alpha_0) ^{t}G(x)\rangle$  are isodual $\alpha$-constacyclic codes of Type III over $\mathcal{R}_2$ for each $G(x)\neq 0$  and for each integer $t \geq 1.$  
 
\vspace{-2mm}\item[(c)] When $\text{char }\mathcal{R}_2=2,$ all the isodual $\alpha$-constacyclic codes of  Type IV over $\mathcal{R}_2$ are given by $\langle (x^n-\alpha_0)^{\omega}+\gamma(x^n-\alpha_0)^t G(x), \gamma (x^n-\alpha_0)^{2^s-\omega} \rangle,$  where  $2^{s-1}< \omega  < 2^{s}$  and $0 \leq t<2^s-\omega$ if $G(x) \neq 0.$  

When  $\text{ char }\mathcal{R}_2$ is an odd prime $p,$ the $\alpha$-constacyclic codes $\langle (x^n-\alpha_0)^{\omega}, \gamma (x^n-\alpha_0)^{p^s-\omega} \rangle,$ $\frac{p^s}{2} < \omega  < p^{s},$ are isodual codes of Type IV over $\mathcal{R}_2.$

When $\text{char }\mathcal{R}_2=p^2,$  the $\alpha$-constacyclic codes $\langle (x^n-\alpha_0)^{\omega}, \gamma (x^n-\alpha_0)^{p^s-\omega} \rangle,$   $\frac{2p^s-p^{s-1}}{2} \leq \omega < p^s,$ are isodual codes of Type IV over $\mathcal{R}_2.$ \vspace{-1mm}
\end{enumerate}
\end{cor}
\begin{proof} Let $\mathcal{C}$ be an $\alpha$-constacyclic code of length $np^s$ over $\mathcal{R}_2.$ For $\mathcal{C}$ to be isodual, we must have $|\mathcal{C}|=|\mathcal{C}^{\perp}|.$
\begin{enumerate}
\vspace{-2mm}\item[(a)] Let $\mathcal{C}$ be of Type II, i.e., $\mathcal{C}=\langle\gamma (x^n-\alpha_0)^\tau  \rangle$ for some $\tau,$ $0 \leq \tau  < p^s.$ By Theorems \ref{thm3} and \ref{dual}, we observe that $\mathcal{C}^{\perp}=\langle(x^n-\alpha_0^{-1}) ^{p^s-\tau},\gamma \rangle,$  $|\mathcal{C}|=p^{mn(p^s-\tau)}$ and $|\mathcal{C}^{\perp}|=p^{mn(p^s+\tau)}.$   Now if the code $\mathcal{C}=\langle \gamma (x^n-\alpha_0)^{\tau}\rangle$ is isodual, then we must have $|\mathcal{C}|=|\mathcal{C}^{\perp}|,$ which gives $\tau=0.$ On the other hand, when $\tau=0,$ we see that the codes  $\mathcal{C}=\langle \gamma \rangle$ and $\mathcal{C}^{\perp}=\langle \gamma \rangle$ are $\mathcal{R}_2$-linearly equivalent. From this, it follows that $\langle \gamma \rangle$ is the only isodual $\alpha$-constacyclic code of Type II over $\mathcal{R}_2.$ 
\vspace{-2mm}\item[(b)] If $\mathcal{C}$ is of Type III, then $\mathcal{C}=\langle (x^n-\alpha_0)^{\omega}+\gamma (x^n-\alpha_0)^t G(x)\rangle,$ where $0 < \omega <  p^s,$ $G(x)$ is either 0 or a unit in $\mathcal{R}_{\alpha,0}$ and $0 \leq t< \omega$ if $G(x) \neq 0.$ Here we shall consider the following two cases separately: (i) $\text{char }\mathcal{R}_2=p$ and (ii) $\text{char }\mathcal{R}_2=p^2.$

(i) First let $\text{char }\mathcal{R}_2=p.$ When $G(x)=0,$ by Theorems \ref{thm3} and \ref{dual}, we have $|\mathcal{C}|=p^{2mn(p^s-\omega)},$ $\mathcal{C}^{\perp}=
\langle (x^n-\alpha_0^{-1})^{p^s-\omega} \rangle$ and  $|\mathcal{C}^\perp|=p^{2mn \omega}.$ Now if the code $\mathcal{C}$ is isodual, then we must have $|\mathcal{C}|=|\mathcal{C}^{\perp}|,$ which gives $p=2$ and $\omega = 2^{s-1}.$ Further, when $p=2$ and $\omega = 2^{s-1},$ we see that the codes $\mathcal{C}(\subseteq \mathcal{R}_{\alpha,0})$ and $\mathcal{C}^{\perp}(\subseteq \widehat{\mathcal{R}_{\alpha,0}})$ are $\mathcal{R}_2$-linearly equivalent. 

Next when $G(x)$ is a unit in $\mathcal{R}_{\alpha,0},$ by Theorems \ref{thm3} and \ref{dual}, we see that $|\mathcal{C}|=p^{2mn(p^s-\omega)},$ 
$\mathcal{C}^{\perp}= \langle(-\alpha_0)^{p^s-\omega}(x^n-\alpha_0^{-1})^{p^s-\omega}-\gamma x^{n\omega-nt-\text{deg }G(x)}(-\alpha_0)^{p^s-2\omega+t}(x^n-\alpha_0^{-1}) ^{p^s-2\omega+t}G^*(x)\rangle$ and $|\mathcal{C}^\perp|=p^{2mn\omega}$ if $p^s-2\omega+t \geq 0 $ and that $|\mathcal{C}|=p^{mn(p^s-t)},$ 
$\mathcal{C}^{\perp}= \langle(-\alpha_0)^{\omega-t}(x^n-\alpha_0^{-1})^{\omega-t}-\gamma x^{n\omega-nt-\text{deg }G(x)}G^*(x), \gamma (x^n-\alpha_0^{-1})^{p^s-\omega}\rangle$ and $|\mathcal{C}^\perp|=p^{mn(p^s+t)}$ if $p^s-2\omega+t < 0 .$ Further, in the case when $p^s-2\omega+t \geq 0,$ for the code $\mathcal{C}$ to be isodual, we must have $|\mathcal{C}|=|\mathcal{C}^{\perp}|,$ which gives $p=2$ and $\omega = 2^{s-1}.$ On the other hand, when $p=2$ and $\omega = 2^{s-1},$ it is easy to see that the codes $\mathcal{C}(\subseteq \mathcal{R}_{\alpha,0})$ and $\mathcal{C}^{\perp}(\subseteq \widehat{\mathcal{R}_{\alpha,0}})$ are $\mathcal{R}_2$-linearly equivalent. Furthermore, when $p^s-2\omega+t < 0,$ for the code $\mathcal{C}$ to be isodual, we must have $|\mathcal{C}|=|\mathcal{C}^{\perp}|,$ which implies that $t=0.$ When $t=0$ and $p=2,$ we see that $\mathcal{C}(\subseteq \mathcal{R}_{\alpha,0})$ and $\mathcal{C}^{\perp}(\subseteq \widehat{\mathcal{R}_{\alpha,0}})$ are $\mathcal{R}_2$-linearly equivalent.

(ii) Next suppose that $\text{char } \mathcal{R}_2=p^2.$ 

When $G(x)=0$ and $\omega \leq p^{s-1},$ by Theorems \ref{thm3} and \ref{dual}, we have $|\mathcal{C}|=p^{2mn(p^s-\omega)},$ $\mathcal{C}^{\perp}=\langle(-\alpha_0)^{p^s-\omega}(x^n-\alpha_0^{-1})^{p^s-\omega}+\gamma x^{np^s-np^{s-1}-\text{deg }M_j(x)}(-\alpha_0)^{p^{s-1}-\omega}(x^n-\alpha_0^{-1})^{p^{s-1}-\omega}M_j^*(x) \rangle$ and  $|\mathcal{C}^\perp|=p^{2mn \omega}.$ Now for the code $\mathcal{C}$ to be isodual, we must have $|\mathcal{C}|=|\mathcal{C}^{\perp}|,$ which gives $p=2$ and $\omega = 2^{s-1}.$ Further,  when $p=2$ and $\omega = 2^{s-1},$ we note that $M_j(x)=za_1\alpha_0^{2^{s-1}}\in \mathcal{R}_2,$ which implies that $\mathcal{C}^{\perp}= \langle(-\alpha_0)^{2^{s-1}}(x^n-\alpha_0^{-1})^{2^{s-1}}+ \gamma x^{n2^{s-1}} za_1\alpha_0^{2^{s-1}} \rangle.$ It is easy to observe that $\mathcal{C}^{\perp}$ is $\mathcal{R}_2$-linearly equivalent to the $\alpha$-constacyclic code $\mathcal{D}=\langle(x^n-\alpha_0)^{2^{s-1}}+ \gamma  za_1\alpha_0^{2^{s-1}}\rangle$ of length $n2^s$ over $\mathcal{R}_2.$ In view of this, we see that the codes $\mathcal{C}$ and $\mathcal{C}^{\perp}$ are $\mathcal{R}_2$-linearly equivalent if and only if $\mathcal{C}$ and $\mathcal{D}$ are $\mathcal{R}_2$-linearly equivalent. For $s=1,$ we see that $\mathcal{C}=\langle x^n-\alpha_0 \rangle $ and $\mathcal{D}=\langle x^n-\alpha_0+ \gamma za_1  \rangle,$ which are trivially $\mathcal{R}_2$-linearly equivalent. For $s\geq 2,$ by Proposition \ref{bino},  we note that $2 ||\binom{2^{s-1}}{2^{s-2}}$ and that $\binom{2^{s-1}}{i}=0$ for each $i$ satisfying $1\leq i \leq 2^{s-1}-1$ and $i \neq 2^{s-2}.$ From this, we get $\mathcal{C}=\langle x^{n2^{s-1}}+\binom{2^{s-1}}{2^{s-2}}x^{n2^{s-2}}(-\alpha_0)^{2^{s-2}}+\alpha_0^{2^{s-1}} \rangle $ and $\mathcal{D}=\langle x^{n2^{s-1}}+\binom{2^{s-1}}{2^{s-2}}x^{n2^{s-2}}(-\alpha_0)^{2^{s-2}}+\alpha_0^{2^{s-1}}(1+ \gamma  za_1)\rangle.$ It is easy to observe that the codes  $\mathcal{C}(\subseteq \mathcal{R}_{\alpha,0})$ and $\mathcal{D}(\subseteq \mathcal{R}_{\alpha,0})$ are $\mathcal{R}_2$-linearly equivalent.

 Next when $G(x) \neq 0,$  $p^s-p^{s-1}+t-\omega > 0$ and $\omega \leq p^{s-1},$ by Theorems \ref{thm3} and \ref{dual}, we have $|\mathcal{C}|=p^{2mn(p^s-\omega)},$  $\mathcal{C}^{\perp}=\langle(-\alpha_0)^{p^s-\omega}(x^n-\alpha_0^{-1})^{p^s-\omega}+\gamma x^{np^s-np^{s-1}-\text{deg }M_j(x)}(-\alpha_0)^{p^{s-1}-\omega}(x^n-\alpha_0^{-1})^{p^{s-1}-\omega}M_j^*(x)-\gamma x^{n\omega-nt-\text{deg }G(x)}(-\alpha_0)^{p^s-2\omega+t }(x^n-\alpha_0^{-1})^{p^s-2\omega+t }G^*(x) \rangle$ and  $|\mathcal{C}^\perp|=p^{2mn \omega}.$  
Here for the code $\mathcal{C}$ to be isodual, we must have $|\mathcal{C}|=|\mathcal{C}^{\perp}|,$ which gives $p=2$ and $\omega = 2^{s-1}.$ On the other hand,  when $p=2$ and $\omega = 2^{s-1},$ we see that $\mathcal{C}^{\perp}= \langle(-\alpha_0)^{2^{s-1}}(x^n-\alpha_0^{-1})^{2^{s-1}} -\gamma (-\alpha_0)^{t} x^{n2^{s-1}-nt-\text{deg }G(x)} (x^n-\alpha_0^{-1})^{t}G^*(x)+\gamma  za_1\alpha_0^{2^{s-1}}x^{n2^{s-1}} \rangle,$ which is $\mathcal{R}_2$-linearly equivalent to the  $\alpha$-constacyclic code $\mathcal{D}_{1}=\langle(x^n-\alpha_0)^{2^{s-1}}-\gamma (x^n-\alpha_0)^{t}G(x)+ \gamma  za_1\alpha_0^{2^{s-1}}\rangle$ of length $n2^s$ over $\mathcal{R}_2.$ Further, one can easily  observe that the codes  $\mathcal{C}(\subseteq \mathcal{R}_{\alpha,0})$ and $\mathcal{D}_{1}(\subseteq \mathcal{R}_{\alpha,0})$ are $\mathcal{R}_2$-linearly equivalent, which implies that the codes $\mathcal{C}(\subseteq \mathcal{R}_{\alpha,0})$ and $\mathcal{C}^{\perp}(\subseteq \widehat{\mathcal{R}_{\alpha,0}})$ are $\mathcal{R}_2$-linearly equivalent.
 
When $p^s-p^{s-1}+t=\omega$ and $\gamma M_j(x)=\gamma G(x),$ by Theorems \ref{thm3} and \ref{dual}, we have $|\mathcal{C}|=p^{2mn(p^s-\omega)},$  $\mathcal{C}^{\perp}=\langle(x^n-\alpha_0^{-1})^{p^s-\omega} \rangle$ and  $|\mathcal{C}^\perp|=p^{2mn \omega}.$ Now for the code $\mathcal{C}$ to be isodual, we must have $p=2$ and $\omega = 2^{s-1}.$ On the other hand, when $p=2$ and $\omega = 2^{s-1},$ it is easy to see that the codes $\mathcal{C}(\subseteq \mathcal{R}_{\alpha,0})$ and $\mathcal{C}^{\perp}(\subseteq \widehat{\mathcal{R}_{\alpha,0}})$ are $\mathcal{R}_2$-linearly equivalent.

\vspace{-2mm}\item[(c)] If $\mathcal{C}$ is of Type IV, then $\mathcal{C} =\langle (x^n-\alpha_0)^{\omega}+\gamma (x^n-\alpha_0)^t G(x), \gamma (x^n-\alpha_0)^{\mu} \rangle,$ where $0 < \mu < \omega <  p^s,$ $G(x)$ is either 0 or a unit in $\mathcal{R}_{\alpha,0}$ and $0 \leq t < \mu$ if $G(x) \neq 0.$ Here by Theorems \ref{thm3} and \ref{dual}, we have $|\mathcal{C}|=p^{mn(2p^s-\omega-\mu)}$ and $|\mathcal{C}^{\perp}|=p^{mn(\omega+\mu)}.$ Now if the code $\mathcal{C}$ is isodual, then $2p^s-\omega-\mu=\omega+\mu,$ which gives $\mu=p^s-\omega.$ 

\vspace{-1mm}Now let $\mu=p^s-\omega.$ Then we have $\mathcal{C}=\langle (x^n-\alpha_0)^{\omega}+\gamma (x^n-\alpha_0)^t G(x), \gamma (x^n-\alpha_0)^{p^s-\omega} \rangle.$ Here also, we shall distinguish the following two cases: (i) $\text{char }\mathcal{R}_2=p$ and (ii) $\text{char }\mathcal{R}_2=p^2.$

(i) First let $\text{char }\mathcal{R}_2=p.$  When $p=2,$ working in a similar manner as in part (b), we see that the codes $\mathcal{C}(\subseteq \mathcal{R}_{\alpha,0})$ and $\mathcal{C}^{\perp}(\subseteq \widehat{\mathcal{R}_{\alpha,0}})$ are $\mathcal{R}_2$-linearly equivalent. When $p$ is an odd prime and $G(x)=0,$ we have $\mathcal{C}=\langle (x^n-\alpha_0)^{\omega}, \gamma (x^n-\alpha_0)^{p^s-\omega} \rangle.$ In this case, by Theorem \ref{dual}, we see that $\mathcal{C}^{\perp}=
\langle (x^n-\alpha_0^{-1})^{\omega}, \gamma (x^n-\alpha_0^{-1})^{p^s-\omega}\rangle ,$ which is  $\mathcal{R}_2$-linearly equivalent to $\mathcal{C}.$  

(ii) Next suppose that $\text{char } \mathcal{R}_2=p^2,$ $G(x)=0$ and $\omega \geq\frac{2p^s-p^{s-1}}{2}.$ Here  by Theorem \ref{dual}, we see that $\mathcal{C}^{\perp}=\langle (x^n-\alpha_0^{-1})^{\omega}, \gamma (x^n-\alpha_0^{-1})^{p^s-\omega}\rangle,$ which is clearly $\mathcal{R}_2$-linearly equivalent to the code $\mathcal{C}.$  \vspace{-1mm}
\end{enumerate}
This completes the proof.\vspace{-2mm}\end{proof}

In the following theorem, we determine Hamming distances of all $\alpha$-constacyclic codes of length $np^s$ over $\mathcal{R}_2$ when $x^n-\alpha_0$ is irreducible over $\mathcal{R}_2.$
\vspace{-2mm}\begin{thm}\label{dthm2} Let $\alpha=\alpha_0^{p^s} \in \mathcal{T}_{2}\setminus \{0\},$ where $\alpha_0 \in \mathcal{T}_{2}$ is such that $x^n-\alpha_0$ is irreducible over $\mathcal{R}_2.$ Let  $\mathcal{C}$ be an $\alpha$-constacyclic code of length $np^s$ over $\mathcal{R}_2$ (as determined  in Theorem \ref{thm2}). Then the Hamming distance $d_H(\mathcal{C})$ of the code $\mathcal{C}$ is given by the following:\begin{enumerate}
\vspace{-2mm}\item[(a)] If $\mathcal{C}=\{ 0 \},$ then  $d_H(\mathcal{C})=0.$
\vspace{-2mm}\item[(b)] If $\mathcal{C}=\langle 1 \rangle,$ then $d_H(\mathcal{C})=1.$
\vspace{-2mm}\item[(c)] If $\mathcal{C}= \langle \gamma (x^{n}-\alpha_0)^{\tau}\rangle$ is of the Type II, then we have
\vspace{-3mm}\begin{equation*}d_H(\mathcal{C})=\left\{\begin{array}{ll}
1 & \text{ if }  \tau=0;\\
\ell+2 & \text{ if } \ell p^{s-1}+1 \leq \tau \leq  (\ell +1)p^{s-1} \text{ with } 0 \leq \ell \leq p-2;\\
(i+1)p^k & \text{ if } p^s-p^{s-k}+(i-1)p^{s-k-1}+1\leq \tau \leq p^s-p^{s-k}+ip^{s-k-1} \\ & \text{ with } 1\leq i \leq p-1 \text{ and } 1 \leq k \leq s-1.\\
\end{array}\right. \vspace{-6mm}\end{equation*}
\vspace{-2mm}\item[(d)] If $\mathcal{C}=\langle (x^{n}-\alpha_0)^{\omega}+\gamma (x^{n}-\alpha_0)^t G(x)\rangle$ is of the Type III, then we have
\vspace{-3mm}\begin{equation*}d_H(\mathcal{C})=\left\{\begin{array}{ll}
\ell+2 & \text{ if } \ell p^{s-1}+1 \leq \kappa \leq  (\ell +1)p^{s-1} \text{ with } 0 \leq \ell \leq p-2;\\
(i+1)p^k & \text{ if } p^s-p^{s-k}+(i-1)p^{s-k-1}+1\leq \kappa \leq p^s-p^{s-k}+ip^{s-k-1} \\ & \text{ with } 1\leq i \leq p-1 \text{ and } 1 \leq k \leq s-1.\\
\end{array}\right. \vspace{-6mm}\end{equation*}
\vspace{-2mm}\item[(e)] If $\mathcal{C}=\langle (x^{n}-\alpha_0)^{\omega}+\gamma (x^{n}-\alpha_0)^t G(x), \gamma(x^{n}-\alpha_0)^\mu\rangle$ is of the Type IV, then we have
\vspace{-3mm}\begin{equation*}d_H(\mathcal{C})=\left\{\begin{array}{ll}
1 & \text{ if }  \mu=0;\\
\ell+2 & \text{ if } \ell p^{s-1}+1 \leq \mu \leq  (\ell +1)p^{s-1} \text{ with } 0 \leq \ell \leq p-2;\\
(i+1)p^k & \text{ if } p^s-p^{s-k}+(i-1)p^{s-k-1}+1\leq \mu \leq p^s-p^{s-k}+ip^{s-k-1} \\ & \text{ with } 1\leq i \leq p-1 \text{ and } 1 \leq k \leq s-1.\\
\end{array}\right. \vspace{-4mm}\end{equation*}
\end{enumerate}
\end{thm}
\begin{proof}  It is easy to see that $d_{H}(\mathcal{C})=0$ when $\mathcal{C}= \{ 0 \},$ and that $d_{H}(\mathcal{C})=1$
 when $\mathcal{C}=\langle 1 \rangle.$   Now to prove (c)-(e), we first observe that \vspace{-2mm} \begin{equation}\label{Q} w_H(Q(x)) \geq w_H(\overline{Q(x)})\end{equation}\vspace{-2mm} for each $Q(x) \in \mathcal{R}_{\alpha, 0} $ satisfying $\overline{Q(x)} \neq 0.$
\begin{description}\vspace{-1mm}\item[(c)] Let $\mathcal{C}= \langle \gamma (x^{n}-\alpha_0)^{\tau}\rangle$ be of the Type II. When $\tau =0,$ we have $\mathcal{C}= \langle \gamma \rangle,$ which implies that $d_H(\mathcal{C})=1.$ Now let $1 \leq \tau \leq p^s-1.$ Here we see that  $d_H(\mathcal{C})$ is equal to the Hamming distance of the $\overline{\alpha_0}^{p^s}$-constacyclic code  $\langle (x^n-\overline{\alpha}_0)^{\tau}\rangle$ of length $np^s$ over $\overline{\mathcal{R}_2}.$ Now by applying Theorem \ref{dthm}, part (c) follows. 

\vspace{-2mm}\item[(d)] Let $\mathcal{C}=\langle (x^{n}-\alpha_0)^{\omega}+\gamma (x^{n}-\alpha_0)^t G(x)\rangle$ be of Type III. Here we observe that an element $Q(x)\in \mathcal{C}$ can be uniquely expressed as $Q(x)= ((x^{n}-\alpha_0)^{\omega}+\gamma (x^{n}-\alpha_0)^t G(x))A_Q(x)+\gamma (x^{n}-\alpha_0)^{\kappa}B_Q(x),$ where $A_Q(x),B_Q(x) \in \mathcal{T}_{2}[x]$ satisfy  $ \text{ deg }A_Q(x) \leq n(p^s-\omega)-1$ if $A_Q(x) \neq 0$ and $ \text{ deg }B_Q(x) \leq n(p^s-\kappa)-1$ if $B_Q(x) \neq 0.$ When $A_Q(x)=0,$ we see that $Q(x)=\gamma B_Q(x) (x^{n}-\alpha_0)^\kappa$ with $B_Q(x) \in \mathcal{T}_{2}[x]\setminus \{0\},$ which implies that $w_H(Q(x)) =w_H( \gamma B_Q(x) (x^{n}-\alpha_0)^\kappa) \geq d_{H}( \langle \gamma(x^{n}-\alpha_0)^\kappa\rangle ) =d_{H}( \langle(x^{n}-\overline{\alpha_0})^\kappa\rangle).$ On the other hand, when  $A_Q(x) \neq 0,$ we see that $Q(x) =(x^{n}-\alpha_0)^{\omega}A_Q(x)+\gamma( (x^{n}-\alpha_0)^t G(x)A_Q(x)+(x^{n}-\alpha_0)^{\kappa}B_Q(x)).$ By \eqref{Q}, we get $w_{H}(Q(x)) \geq w_{H}((x^{n}-\overline{\alpha}_0)^{\omega}\overline{A_Q(x)}) \geq d_{H}( \langle(x^{n}-\overline{\alpha_0})^\omega\rangle).$ Since $\langle(x^{n}-\overline{\alpha_0})^\omega\rangle \subseteq \langle(x^{n}-\overline{\alpha_0})^\kappa\rangle,$ we have $d_{H}( \langle(x^{n}-\overline{\alpha_0})^\omega\rangle) \geq d_{H}( \langle(x^{n}-\overline{\alpha_0})^\kappa\rangle).$ From this, we get $w_H(Q(x)) \geq d_{H}( \langle  (x^{n}-\overline{\alpha_0})^\kappa\rangle)$ for each $Q(x) (\neq 0) \in \mathcal{C}.$ This implies that $d_{H}(\mathcal{C}) \geq d_{H}( \langle  (x^{n}-\overline{\alpha_0})^\kappa\rangle).$

Further, as $\langle \gamma(x^{n}-\alpha_0)^\kappa\rangle \subseteq \mathcal{C},$ we have $d_{H}(\langle \gamma(x^{n}-\alpha_0)^\kappa\rangle) \geq d_{H}(\mathcal{C}).$ This shows that $d_{H}(\langle \gamma(x^{n}-\alpha_0)^\kappa\rangle) = d_{H}(\mathcal{C}).$ Further, we see that $d_{H}(\langle \gamma(x^{n}-\alpha_0)^\kappa\rangle)$ is equal to the Hamming distance of the $\overline{\alpha_0}^{p^s}$-constacyclic code $\langle (x^{n}-\overline{\alpha_0})^\kappa\rangle$ of length $np^s$ over $\overline{\mathcal{R}_{2}}.$ From this and by applying Theorem \ref{dthm}, part (d) follows.

\vspace{-2mm}\item[(e)] 

Let $\mathcal{C}=\langle (x^{n}-\alpha_0)^{\omega}+\gamma (x^{n}-\alpha_0)^t G(x), \gamma(x^{n}-\alpha_0)^\mu\rangle$ be of the Type IV.  When $\mu =0,$ we have $\gamma \in \mathcal{C},$ which implies that $d_{H}(\mathcal{C})=1.$ 

Now let $\mu \geq 1.$ Here we observe that each codeword $Q(x)\in \mathcal{C}$ can be uniquely expressed as $Q(x)= ((x^{n}-\alpha_0)^{\omega}+\gamma (x^{n}-\alpha_0)^t G(x))A_Q(x)+\gamma (x^{n}-\alpha_0)^{\mu}B_Q(x),$ where $A_Q(x),B_Q(x) \in \mathcal{T}_{2}[x]$ satisfy  $ \text{ deg }A_Q(x) \leq n(p^s-\omega)-1$ if $A_Q(x)\neq 0$ and $ \text{ deg }B_Q(x) \leq n(p^s-\mu)-1$ if $B_Q(x)\neq 0.$ When $A_Q(x) =0,$ we have $Q(x)=\gamma (x^{n}-\alpha_0)^\mu B_Q(x),$ which implies that $w_{H}(Q(x)) \geq d_{H}(\langle \gamma  (x^{n}-\alpha_0)^\mu \rangle).$ 

Next when $A_Q(x) \neq 0,$ we have $Q(x)=(x^{n}-\alpha_0)^\omega A_Q(x) +\gamma((x^{n}-\alpha_0)^t A_Q(x)+ (x^{n}-\alpha_0)^\mu B_Q(x)),$ which, by \eqref{Q}, implies that $w_{H}(Q(x)) \geq w_{H}((x^{n}-\overline{\alpha}_0)^{\omega}\overline{A_Q(x)}) \geq d_{H}( \langle(x^{n}-\overline{\alpha_0})^\omega\rangle).$ As $\langle(x^{n}-\overline{\alpha_0})^\omega\rangle \subseteq \langle(x^{n}-\overline{\alpha_0})^\mu\rangle,$ we get $d_{H}( \langle(x^{n}-\overline{\alpha_0})^\omega\rangle) \geq d_{H}( \langle(x^{n}-\overline{\alpha_0})^\mu\rangle).$ From this, we get $w_H(Q(x)) \geq d_{H}( \langle  (x^{n}-\overline{\alpha_0})^\mu\rangle)$ for each $Q(x) (\neq 0) \in \mathcal{C}.$ This implies that $d_{H}(\mathcal{C}) \geq d_{H}( \langle  (x^{n}-\overline{\alpha_0})^\mu\rangle).$

On the other hand, since $\langle \gamma(x^{n}-\alpha_0)^\mu \rangle \subseteq \mathcal{C},$ we have $d_{H}(\langle \gamma(x^{n}-\alpha_0)^\mu\rangle) \geq d_{H}(\mathcal{C}).$ Since $d_{H}( \langle \gamma (x^{n}-\alpha_0)^\mu\rangle)$ is equal to the Hamming distance of the $\overline{\alpha_0}^{p^s}$-constacyclic code  $\langle (x^{n}-\overline{\alpha_0})^\mu\rangle$ of length $np^s$ over $\overline{\mathcal{R}_{2}},$ we obtain $d_{H}(\mathcal{C})=d_H(\langle (x^{n}-\overline{\alpha_0})^\mu\rangle) .$ From this and by applying Theorem \ref{dthm}, part (e) follows.

\end{description}\vspace{-8mm}
\end{proof}

In the following theorem, we determine RT distances of all $\alpha$-constacyclic codes of length $np^s$ over $\mathcal{R}_{2}$ when $x^n-\alpha_0$ is irreducible over $\mathcal{R}_{2}.$
\vspace{-2mm}\begin{thm}   \label{RTD2} Let $\alpha=\alpha_0^{p^s} \in \mathcal{T}_{2}\setminus \{0\},$ where $\alpha_0 \in \mathcal{T}_{2}$ is such that $x^n-\alpha_0$ is irreducible over $\mathcal{R}_2.$ Let  $\mathcal{C}$ be an $\alpha$-constacyclic code of length $np^s$ over $\mathcal{R}_2$ (as determined  in Theorem \ref{thm2}). Then  the RT distance $d_{RT}(\mathcal{C})$ of the code $\mathcal{C}$ is given by the following:
\begin{enumerate} 
\vspace{-1mm}\item[(a)] If $\mathcal{C}=\{0\},$ then $d_{RT}(\mathcal{C})=0.$
\vspace{-1mm}\item[(b)] If $\mathcal{C}=\langle 1 \rangle,$ then $d_{RT}(\mathcal{C})=1.$
\vspace{-1mm}\item[(c)] If $\mathcal{C}=\langle \gamma (x^{n}-\alpha_0)^{\tau}\rangle$ is of the Type II, then $d_{RT}(\mathcal{C})=n\tau +1.$
\vspace{-1mm}\item[(d)] If $\mathcal{C}=\langle (x^{n}-\alpha_0)^{\omega}+\gamma(x^{n}-\alpha_0)^{t}G(x)\rangle$ is of the Type III, then $d_{RT}(\mathcal{C})=n\kappa +1.$
\vspace{-1mm}\item[(e)] If $\mathcal{C}=\langle (x^{n}-\alpha_0)^{\omega}+\gamma(x^{n}-\alpha_0)^{t}G(x), \gamma(x^{n}-\alpha_0)^{\mu} \rangle$ is of the Type III, then $d_{RT}(\mathcal{C})=n\mu +1.$
\end{enumerate}
\vspace{-3mm}\end{thm}
\begin{proof}  One can easily observe that $d_{RT}(\mathcal{C})=0$ when $\mathcal{C}=\{0\},$ and that  $d_{RT}(\mathcal{C})=1$ when  $\mathcal{C}=\langle 1 \rangle.$ 

To prove (c), let $\mathcal{C}=\langle \gamma (x^{n}-\alpha_0)^{\tau}\rangle$ be of the Type II. Here we see that $\mathcal{C}=\langle \gamma (x^{n}-\alpha_0)^{\tau}\rangle =\{ \gamma (x^{n}-\alpha_0)^{\tau}f(x) \  | \  f(x) \in \mathcal{T}_{2}[x]\}.$ For each non-zero $Q(x) \in \mathcal{C},$ we observe that $w_{RT}(Q(x)) \geq w_{RT}(\gamma (x^{n}-\alpha_0)^{\tau})=n\tau+1.$ As $w_{RT}(\gamma (x^{n}-\alpha_0)^{\tau})=n \tau+1,$ we obtain $d_{RT}(\mathcal{C}) = n\tau+1.$ 

To prove (d),  let $\mathcal{C}=\langle (x^{n}-\alpha_0)^{\omega}+\gamma(x^{n}-\alpha_0)^{t}G(x)\rangle$ be  of the Type III.  Here we recall that $\kappa$ is the smallest  integer satisfying $0 < \kappa \leq \omega$ and $\gamma (x^{n}-\alpha_0)^{\kappa} \in \langle (x^{n}-\alpha_0)^{\omega}+\gamma (x^{n}-\alpha_0)^t G(x)\rangle.$ We further note  that $\langle \gamma(x^{n}-\alpha_0)^{\kappa} \rangle = \mathcal{C}\cap \langle  \gamma \rangle.$ Since $\langle \gamma(x^{n}-\alpha_0)^{\kappa} \rangle \subseteq \mathcal{C},$ we have $d_{RT}(\langle \gamma(x^{n}-\alpha_0)^{\kappa}\rangle ) \geq d_{RT}(\mathcal{C}).$ Further, we observe that  $w_{RT}(Q(x)) \geq w_{RT}(\gamma Q(x))$ for each $Q(x) \in \mathcal{C} \setminus \langle \gamma \rangle,$ which implies that $w_{RT}(Q(x)) \geq d_{RT}(\langle \gamma(x^{n}-\alpha_0)^{\kappa}\rangle) $ for each $Q(x) \in \mathcal{C} \setminus \langle \gamma \rangle.$ From this, we obtain $d_{RT}( \mathcal{C}) = d_{RT}(\langle\gamma(x^{n}-\alpha_0)^{\kappa}\rangle).$  This, by part (c),  implies that  $d_{RT}(\mathcal{C})=n\kappa+1.$

To prove (e), let $\mathcal{C}=\langle (x^{n}-\alpha_0)^{\omega}+\gamma(x^{n}-\alpha_0)^{t}G(x), \gamma(x^{n}-\alpha_0)^{\mu}\rangle$ be of the Type IV. Here we note  that $\langle \gamma(x^{n}-\alpha_0)^{\mu} \rangle = \mathcal{C}\cap \langle  \gamma \rangle.$ Since $w_{RT}(Q(x)) \geq w_{RT}(\gamma Q(x))$ for each $Q(x) \in \mathcal{C} \setminus \langle \gamma \rangle,$ we get $w_{RT}(Q(x)) \geq d_{RT}(\langle\gamma(x^{n}-\alpha_0)^{\mu}\rangle) $ for each $Q(x) \in \mathcal{C} \setminus \langle \gamma \rangle.$ This implies that $d_{RT}( \mathcal{C}) \geq d_{RT}(\langle \gamma(x^{n}-\alpha_0)^{\mu}\rangle).$ Further, as $\langle \gamma(x^{n}-\alpha_0)^{\mu} \rangle \subseteq \mathcal{C},$ we have $d_{RT}(\langle\gamma(x^{n}-\alpha_0)^{\mu}\rangle) \geq d_{RT}(\mathcal{C}).$ This implies that $d_{RT}(\mathcal{C})=d_{RT}(\langle\gamma(x^{n}-\alpha_0)^{\mu}\rangle).$ This, by part (c), implies that  $d_{RT}(\mathcal{C})=n\mu+1.$

This completes the proof of the theorem.\vspace{-3mm}\end{proof}
In the following theorem, we determine RT weight distributions of all $\alpha$-constacyclic codes of length $np^s$ over $\mathcal{R}_{2}$ when $x^n-\alpha_0$ is irreducible over $\mathcal{R}_{2}.$
\vspace{-2mm}\begin{thm} \label{RTW2}Let $\alpha=\alpha_0^{p^s} \in \mathcal{T}_{2}\setminus \{0\},$ where $\alpha_0 \in \mathcal{T}_{2}$ is such that $x^n-\alpha_0$ is irreducible over $\mathcal{R}_2.$ Let  $\mathcal{C}$ be an $\alpha$-constacyclic code of length $np^s$ over $\mathcal{R}_2$ (as determined in Theorem \ref{thm2}). For $0 \leq \rho \leq np^s,$ let $\mathcal{A}_{\rho}$ denote the number of codewords in $\mathcal{C}$ having the RT weight as $\rho.$
\begin{enumerate}
\vspace{-2mm}\item[(a)] If $\mathcal{C}=\{0\},$ then we have $\mathcal{A}_0=1$ and $\mathcal{A}_\rho=0$ for $1 \leq \rho \leq np^s.$ \vspace{-2mm}\item[(b)] If $\mathcal{C}=\langle 1 \rangle,$ then $\mathcal{A}_0=1$ and $\mathcal{A}_\rho=(p^{2m}-1)p^{2m(\rho-1)}$ for $1 \leq \rho \leq np^s.$
\vspace{-2mm}\item[(c)] If $\mathcal{C}=\langle \gamma (x^{n}-\alpha_0)^{\tau}\rangle$ is of the Type II, then we have\vspace{-3mm} \begin{equation*}\mathcal{A}_\rho=\left\{\begin{array}{ll}
1 & \text{ if } \rho=0;\\
0 & \text{ if } 1\leq \rho \leq n\tau;\\
(p^{m}-1)p^{m(\rho-n\tau-1)}  & \text{ if } n\tau+1 \leq \rho \leq np^s.
\end{array}\right. \vspace{-3mm}\end{equation*}

\vspace{-2mm}\item[(d)] If $\mathcal{C}=\langle (x^{n}-\alpha_0)^{\omega}+\gamma(x^{n}-\alpha_0)^{t}G(x)\rangle$ is of the Type III, then we have \vspace{-4mm}\begin{equation*}\mathcal{A}_\rho=\left\{\begin{array}{ll}
1 & \text{ if } \rho=0;\\
0 & \text{ if } 1\leq \rho \leq n\kappa;\\
(p^{m}-1)p^{m(\rho-n\kappa-1)}  & \text{ if } n\kappa+1 \leq \rho \leq n\omega;\\
(p^{2m}-1)p^{m(2\rho-n\omega-n\kappa-2)} & \text{ if } n\omega+1 \leq \rho \leq np^s. \end{array}\right. \vspace{-3mm}\end{equation*}

\vspace{-2mm}\item[(e)] If $\mathcal{C}=\langle (x^{n}-\alpha_0)^{\omega}+\gamma(x^{n}-\alpha_0)^{t}G(x), \gamma(x^{n}-\alpha_0)^{\mu} \rangle$ is of the Type IV, then we have \vspace{-4mm}\begin{equation*}\mathcal{A}_\rho=\left\{\begin{array}{ll}
1 & \text{ if } \rho=0;\\
0 & \text{ if } 1\leq \rho \leq n\mu;\\
(p^{m}-1)p^{m(\rho-n\mu-1)}  & \text{ if } n\mu+1 \leq \rho \leq n\omega;\\
(p^{2m}-1)p^{m(2\rho-n\omega-n\mu-2)} & \text{ if } n\omega+1 \leq \rho \leq np^s. \end{array}\right. \vspace{-3mm}\end{equation*}
\end{enumerate}
\end{thm}

\begin{proof} Proofs of parts (a) and (b) are trivial. 

To prove (c), let $\mathcal{C}= \langle \gamma (x^{n}-\alpha_0)^{\tau}\rangle$ be of the Type II. Note that $\mathcal{A}_0=1.$ By Theorem \ref{RTD2}(c), we see that $d_{RT}(\mathcal{C})=n\tau+1,$ which implies that $\mathcal{A}_\rho=0$ for $1 \leq \rho \leq n\tau.$ Next let $n \tau +1 \leq \rho \leq np^s.$ Here we see that $\mathcal{C}= \langle \gamma (x^{n}-\alpha_0)^{\tau}\rangle=\{ \gamma (x^{n}-\alpha_0)^{\tau} F(x)\ |\ F(x) \in \mathcal{T}_{2}[x]\}.$  From this, we see that the codeword $\gamma (x^{n}-\alpha_0)^{\tau} F(x) \in \mathcal{C}$ has RT weight $\rho$ if and only if   $\text{ deg } F(x)=\rho-n\tau-1.$ This implies that $\mathcal{A}_\rho=(p^{m}-1)p^{m(\rho-n\tau-1)} $ for $n\tau+1 \leq \rho \leq np^s.$

To prove (d), let $\mathcal{C}=\langle (x^{n}-\alpha_0)^{\omega}+\gamma (x^{n}-\alpha_0)^t G(x)\rangle$ be of the Type III. Here by Theorem \ref{RTD2}, we see that  $d_{RT}(\mathcal{C})=n\kappa+1,$ which implies that $\mathcal{A}_\rho=0$ for $1\leq \rho \leq n\kappa.$ 
Next let $n\kappa + 1\leq \rho \leq n p^s.$ Here it is easy to observe that an element $Q(x)\in \mathcal{C}$ can be uniquely expressed as $Q(x)= ((x^{n}-\alpha_0)^{\omega}+\gamma (x^{n}-\alpha_0)^t G(x))A_Q(x)+\gamma (x^{n}-\alpha_0)^{\kappa}B_Q(x),$ where $A_Q(x),B_Q(x) \in \mathcal{T}_{2}[x]$ satisfy  $ \text{ deg }A_Q(x) \leq n(p^s-\omega)-1$ if $A_Q(x) \neq 0$ and $ \text{ deg }B_Q(x) \leq n(p^s-\kappa)-1$ if $B_Q(x) \neq 0.$   From this, we see that if $ n\kappa+1 \leq \rho \leq n\omega,$ then the RT weight of the codeword $Q(x) \in \mathcal{C}$ is $\rho$ if and only if  $A_Q(x)=0$ and $\text{ deg }B_Q(x)=\rho-n\kappa-1.$ This implies that $\mathcal{A}_\rho=(p^{m}-1)p^{m(\rho-n\kappa-1)}$ for $ n\kappa+1 \leq \rho \leq n\omega.$
Further, if $n\omega+1 \leq \rho \leq np^s,$ then the RT weight of the codeword $Q(x) \in \mathcal{C}$ is $\rho$ if and only if one of the following two conditions are satisfied: (i)  $\text{ deg }A_Q(x)=\rho-n\omega-1$ and $B_Q(x)$ is either 0 or $\text{deg }B_Q(x) \leq \rho-1-n\kappa$ and  (ii) $A_Q(x)$ is either 0 or $\text{deg }A_Q(x) \leq \rho-n \omega-2$ and $\text{deg }B_Q(x) =\rho-n\kappa-1.$ From this, we get $\mathcal{A}_\rho=(p^{2m}-1)p^{m(2\rho-n\omega-n\kappa-2)}$ for $n\omega+1 \leq \rho \leq np^s.$

To prove (e),  let $\mathcal{C}=\langle (x^{n}-\alpha_0)^{\omega}+\gamma (x^{n}-\alpha_0)^t G(x), \gamma (x^{n}-\alpha_0)^\mu \rangle$ be of the Type IV.  By Theorem \ref{RTD2}, we see that $d_{RT}(\mathcal{C})=n\mu+1,$ which implies that $\mathcal{A}_\rho=0$ for $1\leq \rho \leq n\mu.$ Next let $n \mu + 1 \leq \rho \leq np^s.$ Here we see that each codeword $Q(x)\in \mathcal{C}$ can be uniquely expressed as $Q(x)= ((x^{n}-\alpha_0)^{\omega}+\gamma (x^{n}-\alpha_0)^t G(x))M_Q(x)+\gamma (x^{n}-\alpha_0)^{\mu}W_Q(x),$ where $M_Q(x),W_Q(x) \in \mathcal{T}_{2}[x]$ satisfy  $ \text{ deg }M_Q(x) \leq n(p^s-\omega)-1$ if $M_Q(x)\neq 0$ and $ \text{ deg }W_Q(x) \leq n(p^s-\mu)-1$ if $W_Q(x)\neq 0.$   From this, we see that if $n\mu+1 \leq \rho \leq n\omega,$ then the codeword $Q(x) \in \mathcal{C}$ has  RT weight $\rho$ if and only if $M_Q(x)=0$ and $\text{ deg }W_Q(x)=\rho-n\mu-1.$ This implies that $\mathcal{A}_\rho=(p^{m}-1)p^{m(\rho-n\mu-1)}$ for $ n\mu+1 \leq \rho \leq n\omega.$
Further, if $n\omega+1 \leq \rho \leq np^s,$ then the RT weight of the codeword $Q(x) \in \mathcal{C}$ is $\rho$ if and only if  one of the following two conditions are satisfied: (i) $\text{ deg }M_Q(x)=\rho-n\omega-1$ and $W_Q(x)$ is either 0 or $\text{deg }W_Q(x) \leq \rho-1-n\mu$   and (ii) $M_Q(x)$ is either 0 or $\text{deg }M_Q(x) \leq \rho-n\omega-2$ and $\text{ deg }W_Q(x) =\rho-n\mu-1.$ This implies that $\mathcal{A}_\rho=(p^{2m}-1)p^{m(2\rho-n\omega-n\mu-2)}.$
 
This completes the proof of the theorem.\vspace{-2mm}\end{proof}

To illustrate our results, we shall determine all  cyclic and negacyclic codes of length $10$ over the Galois ring $GR(4,3)$ of characteristic $4$ and cardinality $4^3.$

\vspace{-1mm}\begin{ex} In order to write down all negacyclic codes of length $10$ over $GR(4,3),$  we observe that  an irreducible factorization of $x^5-1$ over $GR(4,3) $ is given by $x^5-1=(x+3)(x^4+x^3+x^2+x+1).$ Further, working as in the proof of  Lemma \ref{fac}, we see that $x^{10}+1=\{(x+3)^2+2 (x^5-2)\}\{(x^4+x^3+x^2+x+1)^2+2 (x^5-2)(3x^6+2x^4+x^2+2x)\}$ is a factorization of $x^{10}+1$ into coprime polynomials over $GR(4,3).$ Now by applying the Chinese Remainder Theorem, we get
 $\mathcal{R}_{1,1}=GR(4,3)[x]/\langle x^{10}+1 \rangle \simeq  \mathcal{K}_1 \oplus \mathcal{K}_2,$ where $\mathcal{K}_1= GR(4,3)[x]/\langle (x+3)^2+2 (x^5-2)\rangle$ and $\mathcal{K}_2= GR(4,3)[x]/\langle (x^4+x^3+x^2+x+1)^2+2 (x^5-2)(3x^6+2x^4+x^2+2x)\rangle.$ By Theorem \ref{thm1}, we note that all the ideals of $\mathcal{K}_1$ are given by $\langle (x+3)^i\rangle,$ $0 \leq i \leq 4$ and all the ideals of $\mathcal{K}_2$ are given by $\langle (x^4+x^3+x^2+x+1)^\ell\rangle,$ $0 \leq \ell \leq 4.$ From this and by applying Proposition \ref{p1}, we see that all negacyclic codes of length $10$ over $GR(4,3)$ are given by $\langle (x+3)^i\rangle \oplus \langle (x^4+x^3+x^2+x+1)^\ell\rangle,$ where $0 \leq i, \ell \leq 4.$ By Corollary \ref{c1}, we see that  the code $\langle 2 \rangle $ is a self-dual negacyclic code of length 10 over $GR(4,3).$ 
 \end{ex}
\vspace{-2mm} \begin{ex} Next we proceed to write down all cyclic codes of length $10$ over $GR(4,3),$ which are ideals of the ring $\mathcal{R}_{1,0}=GR(4,3)[x]/\langle x^{10}-1 \rangle.$  To do this, working as in the proof of Lemma \ref{fac}, we see that $x^{10}-1= \{(x+3)^2+2 (x^5-1)\}\{ (x^4+x^3+x^2+x+1)^2+2 (x^5-1)(3x^6+2x^4+x^2+2x)\}$ is a factorization of $x^{10}-1$ into coprime polynomials over $GR(4,3).$ Now by applying the Chinese Remainder Theorem, we obtain $\mathcal{R}_{1,0} \simeq \mathcal{K}_1 \oplus \mathcal{K}_2,$ where $\mathcal{K}_1=GR(4,3)[x]/\langle (x+3)^2+2 (x^5-1)\rangle $ and $\mathcal{K}_2=GR(4,3)[x]/\langle (x^4+x^3+x^2+x+1)^2+2 (x^5-1)(3x^6+2x^4+x^2+2x)\rangle.$
  Further, by applying Proposition \ref{p1}, all cyclic codes of length 10 over $GR(4,3)$ are given by $\mathcal{I}_{1}\oplus \mathcal{I}_{2},$ where $\mathcal{I}_{1}$ is an ideal of $\mathcal{K}_{1}$ and $\mathcal{I}_{2}$ is an ideal of $\mathcal{K}_{2}.$
If $\mathcal{T}_{2}=\{0,1,\xi, \xi^2\}$ is the Teichm\"{u}ller set of $GR(4,3),$ then by applying Theorem \ref{thm2}, we list all the ideals of $\mathcal{K}_1$ and $\mathcal{K}_{2}$ in Tables  ~\ref{table1} and \ref{table2}, respectively.  In Table~\ref{table3}, we list some self-dual cyclic codes of length 10 over $GR(4,3)$ by applying Corollary \ref{c2}.
 \vspace{-1mm}
\begin{table}[h]
\begin{minipage}{\textwidth}  
\centering 
\begin{tabular}{ |l|l|l|l|} 
\hline
 Trivial ideals &  Principal ideals & Non-principal ideals   \\ [0.5ex]
\hline 
$\{ 0\},$ $\mathcal{K}_1$ & $\langle 2\rangle,$ $\langle x+3 \rangle,$ $\langle 2x+2 \rangle,$ $\langle x+1 \rangle,$ $\langle x+3+2\xi \rangle,$  $\langle x+3+2 \xi^2\rangle$  & $\langle x+3, 2 \rangle$ \\ \hline
 \end{tabular}
\end{minipage}
\vspace{-3mm}\caption[Ideals of $\mathcal{K}_1$]{Ideals of $\mathcal{K}_1$}\label{table1}
\end{table}
\vspace{-7mm}\begin{table}[h] 
 \begin{minipage}{\textwidth} \centering
\begin{tabular}{ |l|l|l|l|} 
\hline
 Trivial ideals & Principal ideals & Non-principal ideals  \\ [0.5ex]
\hline 
$\{ 0\},$ $\mathcal{K}_2$  &$\langle 2\rangle,$ $\langle x^4+x^3+x^2+x+1 \rangle,$ $\langle 2x^4+2x^3+2x^2+2x+2\rangle$ & $\langle x^4+x^3+x^2+x+1, 2 \rangle$ \\
&  $\langle x^4+x^3+x^2+x+1+2G(x) \rangle$
  & \\ \hline
\end{tabular}\vspace{-3mm}\caption[Ideals of $\mathcal{K}_2$]{$^a$Ideals of $\mathcal{K}_2$}\label{table2}
\end{minipage}
\end{table}
 \vspace{-7mm}
\begin{table}[h]
\begin{minipage}{\textwidth}  
\begin{tabular}{ |l|l|l|}
\hline $\langle 2\rangle$ & $\langle 2\rangle\oplus \langle x^4+x^3+x^2+x+1 \rangle$ & $\langle 2\rangle\oplus \langle x^4+x^3+x^2+x+1+2G(x) \rangle$\\ \hline
$\langle x+3 \rangle\oplus \langle 2 \rangle$ & $\langle x+3 \rangle\oplus \langle x^4+x^3+x^2+x+1 \rangle$ & $\langle x+3 \rangle\oplus \langle x^4+x^3+x^2+x+1+2G(x) \rangle$ \\ \hline
$\langle x+1 \rangle\oplus \langle 2 \rangle$ & $\langle x+1 \rangle\oplus \langle x^4+x^3+x^2+x+1 \rangle$ & $\langle x+1 \rangle\oplus \langle x^4+x^3+x^2+x+1+2G(x) \rangle$\\ \hline
$\langle  x+3+2\xi \rangle\oplus \langle 2 \rangle$ & $\langle  x+3+2\xi \rangle\oplus \langle x^4+x^3+x^2+x+1 \rangle$ & $\langle  x+3+2\xi \rangle\oplus \langle x^4+x^3+x^2+x+1+2G(x) \rangle$\\ \hline
$\langle  x+3+2\xi^2 \rangle\oplus \langle 2 \rangle$ & $\langle  x+3+2\xi^2 \rangle\oplus \langle x^4+x^3+x^2+x+1 \rangle$ & $ \langle  x+3+2\xi^2 \rangle\oplus \langle x^4+x^3+x^2+x+1+2G(x) \rangle$\\
\hline
\end{tabular}\vspace{-3mm}\centering \caption[Some self-dual cyclic codes of length 10 over $GR(4,3)$]{\footnote{Here $G(x)$ runs over $\mathcal{P}_{4}(\mathcal{T}_{2})\setminus \{0\}.$}Some self-dual cyclic codes of length 10 over $GR(4,3)$}\label{table3}
\end{minipage}
\end{table}
\end{ex}

\vspace{-7mm}\section{Some more results on Hamming distances, RT distances and RT weight distributions of constacyclic codes over $\mathcal{R}_{e}$}\label{sec4}
Throughout this section, let $e \geq 2$ be an integer and let $\mathcal{R}_{e}$ be a finite commutative chain ring with nilpotency index $e.$ Let $\lambda= \theta+\gamma \omega,$ where $\theta \in \mathcal{T}_{e}$ and  $\omega$ is a unit in $\mathcal{R}_{e}.$ As $\theta \in \mathcal{T}_{e},$ by Proposition \ref{teich}(c), there exists $\lambda_0 \in \mathcal{T}_{e}$ satisfying $\theta=\lambda_0^{p^s}.$ 
 In this section, we shall determine the algebraic structures, Hamming distances,   RT distances,  and RT weight distributions of all $(\lambda_0^{p^s}+\gamma \omega)$-constacyclic codes of length $np^s$ over $\mathcal{R}_{e},$ where $\lambda_0 \in \mathcal{T}_{e}$ is such that the polynomial $x^n-\lambda_0$ is irreducible over $\mathcal{R}_{e}$ and  $\omega$ is a unit in $\mathcal{R}_{e}.$ 

 To do this, we see that Dinh et al. \cite[Th. 3.18]{dinh6} determined algebraic structures of all $\lambda$-constacyclic codes of length $p^s$ over $\mathcal{R}_{e}.$ In the following theorem, we extend this result to $\lambda$-constacyclic codes of length $np^s$ over $\mathcal{R}_{e},$ where $n$ is any positive integer.
\vspace{-2mm}\begin{thm}  \label{6thm1} \begin{enumerate}\item[(a)] The ring $\mathcal{R}_{\lambda} =\mathcal{R}_{e}[x]/\langle x^{np^s}-\lambda \rangle$ is a finite commutative chain ring with the unique maximal ideal as $\langle x^n-\lambda_0 \rangle.$\vspace{-2mm} \item[(b)] In the ring $\mathcal{R}_\lambda,$ we have $\langle (x^n-\lambda_0)^{p^s} \rangle= \langle \gamma \rangle,$ and the nilpotency index of $x^n-\lambda_0$ is $ep^s.$ 
\vspace{-2mm}\item[(c)] All the $\lambda$-constacyclic codes of length $np^s$ over $\mathcal{R}_{e}$ are given by $\langle (x^n-\lambda_0)^{\nu} \rangle,$ where $0 \leq \nu \leq ep^s.$  Moreover, for $0 \leq \nu \leq ep^s,$ the code $\langle (x^n-\lambda_0)^{\nu} \rangle$ has  $ p^{mn (ep^s-\nu)}$ codewords and its  dual code  $\langle (x^n-\lambda_0)^{\nu}\rangle$ is given by $\langle (x^n-\lambda_0^{-1})^{ep^s-\nu}\rangle.$\end{enumerate}
\vspace{-2mm}\end{thm}
\vspace{-3mm}\begin{proof} Working in a similar manner as in Theorem 3.18 of Dinh et al. \cite{dinh6}, the result follows.
\vspace{-2mm}\end{proof}
  
In the following theorem, we determine Hamming distances of all $\lambda$-constacyclic codes of length $np^s$ over $\mathcal{R}_{e}.$ It generalizes Theorem 4.3 of Dinh et al. \cite{dinh8}.
  
\vspace{-2mm}\begin{thm}\label{6thm2} Let $\mathcal{C}=\langle (x^n-\lambda_0)^{\nu} \rangle$ be a $\lambda$-constacyclic code of length $np^s$ over $\mathcal{R}_{e},$ where $0 \leq \nu \leq ep^s.$ Then the Hamming distance $d_H(\mathcal{C})$ of  $\mathcal{C}$ is given by 
\vspace{-2mm}\begin{equation*}d_H(\mathcal{C})=\left\{\begin{array}{ll}
1 & \text{ if } 0 \leq \nu \leq (e-1)p^s;\\
\ell+2 & \text{ if } (e-1)p^s+\ell p^{s-1}+1 \leq \nu \leq (e-1)p^s+(\ell +1)p^{s-1} \text{ with } 0 \leq \ell \leq p-2;\\
(i+1)p^k & \text{ if } ep^s-p^{s-k}+(i-1)p^{s-k-1}+1\leq \nu \leq ep^s-p^{s-k}+ip^{s-k-1} \\ & \text{ with } 1\leq i \leq p-1 \text{ and } 1 \leq k \leq s-1;\\
0  & \text{ if } \nu=ep^s.
\end{array}\right. \vspace{-2mm}\end{equation*}
\end{thm}
\begin{proof} Working in a similar way as in Theorem \ref{dthm1}, the result follows.
\vspace{-2mm}\end{proof}

In the following theorem, we determine RT  distances of all $\lambda$-constacyclic codes of length $np^s$ over $\mathcal{R}_{e}.$ It generalizes Theorem 6.2 of Dinh et al. \cite{dinh8}.

\vspace{-2mm}\begin{thm}\label{6thm3} Let $\mathcal{C}=\langle (x^n-\lambda_0)^{\nu} \rangle$ be a $\lambda$-constacyclic code of length $np^s$ over $\mathcal{R}_{e},$ where $0 \leq \nu \leq ep^s.$ Then the RT distance $d_{RT}(\mathcal{C})$ of  $\mathcal{C}$ is given by
\vspace{-2mm}\begin{equation*}d_{RT}(\mathcal{C})=\left\{\begin{array}{ll}
1 & \text{ if } 0 \leq \nu \leq (e-1)p^s;\\
n\nu-n(e-1)p^{s}+1  & \text{ if } (e-1)p^s+1 \leq \nu \leq ep^s-1;\\
0  & \text{ if } \nu =ep^s.
\end{array}\right. \vspace{-2mm}\end{equation*}
\end{thm}
\begin{proof} Working in a similar manner as in Theorem \ref{dRthm1}, the result follows.
\vspace{-2mm}\end{proof}

In a recent work, Dinh et al. \cite[Prop. 6.3-6.5]{dinh8} determined RT weight distributions of all  $(4z-1)$-constacyclic codes of length $2^s$ over the Galois ring $GR(2^e, m)$ of characteristic $2^e$ and cardinality $2^{em},$ where $z$ is a unit in $GR(2^e,m).$  However, we noticed an error in Proposition 6.5 of Dinh et al. \cite{dinh8}, which we illustrate in the following example. 
\vspace{-1mm}\begin{ex}\label{ex} Let $GR(4,1)$ be  the Galois ring of characteristic $4$ and cardinality $4,$ and let $\lambda=3.$ By Theorem 3.3 of Dinh et al.\cite{dinh8}, we see that all the $3$-constacyclic codes of length $4$ over $GR(4,1)$ are ideals  of $GR(4,1)[x]/\langle x^4-3 \rangle,$ and are given by $\langle (x+1)^{i} \rangle ,$ where $ 0 \leq i \leq 8.$ For the code $\mathcal{C}_{1}=\langle x+1 \rangle,$ by Proposition 6.5 of Dinh et al. \cite{dinh8}, we obtain  $\mathcal{A}_0=1,$ $\mathcal{A}_1=1,$ $\mathcal{A}_2=18,$ $\mathcal{A}_3=36$ and $\mathcal{A}_4=72.$ However, by carrying out  computations in Magma, we see that the actual values of $\mathcal{A}_{2},$ $\mathcal{A}_{3}$ and $\mathcal{A}_{4}$ are given by  $\mathcal{A}_2=6,$ $\mathcal{A}_3=24$ and $\mathcal{A}_4=96,$ which do not agree with Proposition 6.5 of  Dinh et al. \cite{dinh8}.

Moreover, for the code $\mathcal{C}_2=\langle (x+1)^{2} \rangle,$ by Proposition 6.5 of Dinh et al. \cite{dinh8}, we obtain  $\mathcal{A}_0=1,$ $\mathcal{A}_1=1,$ $\mathcal{A}_2=2,$  $\mathcal{A}_3=20$ and $\mathcal{A}_4=40.$ However, by carrying out computations in Magma, we see that the actual values of $\mathcal{A}_{3}$ and $\mathcal{A}_{4}$ are given by $\mathcal{A}_3=12$ and $\mathcal{A}_4=48.$ This shows that there is an error in Proposition 6.5. of Dinh et al. \cite{dinh8}.\end{ex}

In the following theorem, we shall rectify the error in Proposition 6.5 of Dinh et al. \cite{dinh8} and we shall extend Propositions 6.3-6.5 of Dinh et al. \cite{dinh8} to determine RT weight distributions of all $\lambda$-constacyclic codes of length $np^s$ over $\mathcal{R}_{e}.$ 

\vspace{-2mm}\begin{thm} \label{6thmRT} Let $\mathcal{C}=\langle (x^n-\lambda_0)^{\nu} \rangle$ be a $\lambda$-constacyclic code of length $np^s$ over $\mathcal{R}_{e},$ where $0 \leq \nu \leq ep^s.$ For $0 \leq \rho \leq np^s,$ let $\mathcal{A}_\rho$ denote the number of codewords in $\mathcal{C}$ having the  RT weight as $\rho.$  
\begin{enumerate}
\vspace{-2mm}\item[(a)] For $\nu=ep^s,$ we have 
\vspace{-2mm}\begin{equation*}\mathcal{A}_\rho=\left\{\begin{array}{ll}
1 & \text{ if } \rho=0;\\
0  & \text{ otherwise}.
\end{array}\right. \vspace{-2mm}\end{equation*}
\vspace{-4mm}\item[(b)] For $(e-1)p^s+1 \leq \nu \leq ep^s-1,$ we have 
\vspace{-3mm}\begin{equation*}\mathcal{A}_\rho=\left\{\begin{array}{ll}
1 & \text{ if } \rho=0;\\
0  & \text{ if } 1 \leq \rho \leq n\nu-n(e-1)p^s ;\\
(p^m-1)p^{m(\rho-n\nu+n(e-1)p^s-1)}  & \text{ if }  n\nu-n(e-1)p^s+1 \leq \rho \leq np^s.
\end{array}\right. \vspace{-2mm}\end{equation*}
\vspace{-4mm}\item[(c)] For $ \nu =y p^s$ with $0 \leq y \leq e-1,$  we have 
\vspace{-3mm}\begin{equation*}\mathcal{A}_\rho=\left\{\begin{array}{ll}
1 & \text{ if } \rho=0;\\
(p^{m(e-y)}-1)p^{m(e-y)(\rho-1)}  & \text{ if } 1 \leq \rho \leq np^s.
\end{array}\right. \vspace{-2mm}\end{equation*}
\vspace{-4mm}\item[(d)] For $(b-1)p^s+1 \leq \nu \leq b p^s-1$ with $1 \leq b \leq e-1,$ we have 
\vspace{-3mm}\begin{equation*}\mathcal{A}_\rho=\left\{\begin{array}{ll}
1 & \text{ if } \rho=0;\\
(p^{m(e-b)}-1)p^{m(e-b)(\rho-1)}  & \text{ if } 1 \leq \rho \leq n\nu-n(b-1)p^s ;\\
p^{m(np^s(b-1)-n\nu-e+b-1)}(p^{m(e-b+1)}-1)p^{m(e-b+1)\rho}  & \text{ if }  n\nu-n(b-1)p^s+1 \leq \rho \leq np^s.
\end{array}\right. \vspace{-2mm}\end{equation*}
\end{enumerate}
\vspace{-3mm}\end{thm}

\begin{proof}  It is easy to see that $\mathcal{A}_{0}=1.$ So from now on, we assume that $1 \leq \rho \leq np^s.$
\begin{enumerate} 
\vspace{-2mm}\item[(a)] When $\nu =ep^s,$ we have $\mathcal{C}=\{ 0\}.$ From this, we get $\mathcal{A}_\rho=0$ for $1 \leq \rho \leq n p^s.$ 
\vspace{-2mm}\item[(b)] In this case, by Theorem \ref{6thm3}, we see that $d_{RT}(\mathcal{C})=n\nu-n(e-1)p^{s}+1,$ which implies that $\mathcal{A}_{\rho}=0$ for $1 \leq \rho \leq n\nu -n(e-1)p^s.$ Next let $n\nu-n(e-1)p^s+1 \leq \rho \leq np^s.$ Here by Theorem \ref{6thm1}(b), we note that $\langle (x^n-\lambda_0)^{p^s} \rangle= \langle \gamma \rangle,$ which implies that $\mathcal{C}=\langle \gamma^{e-1}(x^n-\lambda_0)^{\nu-(e-1)p^s} \rangle=\{\gamma^{e-1}(x^n-\lambda_0)^{\nu-(e-1)p^s} F(x) \ | \ F(x) \in \mathcal{T}_{e}[x]\}.$  From this, we see that the RT weight of the codeword $\gamma^{e-1}(x^n-\lambda_0)^{\nu-(e-1)p^s} F(x) \in \mathcal{C}$ is $\rho$ if and only if  $\text{ deg }F(x)=\rho-n\nu+n(e-1)p^s-1.$ This implies that $\mathcal{A}_\rho=(p^m-1)p^{m(\rho-n\nu+n(e-1)p^s-1)}.$
\vspace{-2mm}\item[(c)] Next let $ \nu =y p^s,$ where $0 \leq y \leq e-1.$ By Theorem \ref{6thm1}(b), we see that $\mathcal{C}=\langle (x^n-\lambda_0)^{yp^s} \rangle=\langle \gamma^{y}\rangle= \{\gamma^{y} F(x) \ | \ F(x) \in \mathcal{R}_e[x]\}.$ From this, we see that $\mathcal{A}_\rho=(p^{m(e-y)}-1)p^{m(e-y)(\rho-1)}$ for $1 \leq \rho \leq np^s.$
\vspace{-2mm}\item[(d)] Finally, let $(b-1)p^s+1 \leq \nu \leq b p^s-1,$ where $1 \leq b \leq e-1.$ As $\langle (x^n-\lambda_0)^{p^s} \rangle= \langle \gamma \rangle,$ we see that $\mathcal{C}=\langle \gamma^{b-1}(x^n-\lambda_0)^{\nu-(b-1)p^s} \rangle$ and $\langle \gamma^{b} \rangle \subseteq \mathcal{C}.$ Further, we observe that each element $Q(x) \in \mathcal{C}$ can be uniquely written as $Q(x) =\gamma^{b-1}(x^n-\lambda_0)^{\nu-(b-1)p^s}F_{Q}(x) +\gamma^{b} H_Q(x),$ where $H_{Q}(x) \in \mathcal{R}_e[x]$ and $F_{Q}(x) \in \mathcal{T}_{e}[x]$ is either 0 or  $\text{ deg }F_Q(x) \leq n(bp^s-\nu)-1.$  

First let $1 \leq \rho \leq n\nu-n(b-1)p^s.$ Here we see that the RT weight of the codeword $Q(x)$ is $\rho$ if and only if $F_Q(x)=0$ and $\text{ deg }H_Q(x)=\rho-1.$ From this, we obtain $\mathcal{A}_\rho=(p^{m(e-b)}-1)p^{m(e-b)(\rho-1)}.$

 Next let $n\nu-n(b-1)p^s+1 \leq \rho \leq np^s.$ In this case, we see that the RT weight of the codeword $Q(x)$ is $\rho$ if and only if one of the following two conditions is satisfied: (i) $\text{ deg }F_Q(x)=\rho-n\nu+n(b-1)p^s-1$ and $H_Q(x)$ is either 0 or $\text{deg }H_Q(x) \leq \rho-1,$ and (ii) $F_Q(x)$ is either 0 or $\text{ deg }F_Q(x)\leq \rho-n\nu+n(b-1)p^s-2$ and $\text{ deg }H_Q(x) =\rho-1.$ From this, we obtain \vspace{-2mm}\begin{eqnarray*}\mathcal{A}_\rho & =& (p^m-1)p^{m(\rho-n\nu+n(b-1)p^s-1)}p^{m(e-b)\rho}+p^{m(\rho-n\nu+n(b-1)p^s-1)}(p^{m(e-b)}-1)p^{m(e-b)(\rho-1)}\\& =&p^{m(np^s(b-1)-n\nu-e+b-1)}(p^{m(e-b+1)}-1)p^{m(e-b+1)\rho}.\end{eqnarray*}
\vspace{-12mm}\end{enumerate}
This completes the proof of the theorem.\vspace{-2mm}\end{proof}

\end{document}